\documentclass[preprint,12pt]{elsarticle}
\journal{European Journal of Combinatorics}
\usepackage{amsmath, amsthm,amssymb,latexsym}
\usepackage{enumerate,caption,xcolor,graphicx,booktabs,multirow,setspace}
\usepackage{dsfont}
\usepackage{soul}

\newtheorem{defn}{Definition}
\newtheorem{thm}{Theorem}[section]
\newtheorem{prop}[thm]{Proposition}
\newtheorem{clm}[thm]{Claim}
\newtheorem{subclm}[thm]{Subclaim}
\newtheorem{cor}[thm]{Corollary}
\newtheorem{lem}[thm]{Lemma}
\newtheorem{problem}[thm]{Problem}

\numberwithin{equation}{section}

\def\diam{\mathrm{diam}}
\def\dist{\mathrm{dist}}

\newcommand{\s}{\mathrm{s}}
\newcommand{\cs}{\mathrm{cs}}
\newcommand{\G}{\mathcal{G}^{\mathrm{cs}}}

\begin{document}
\begin{frontmatter}
\title{Stable structure on safe set problems in vertex-weighted graphs}
\author[yo]{Shinya Fujita}
\author[a]{Boram Park}
\author[ya]{Tadashi Sakuma}
\address[yo]{School of Data Science, Yokohama City University, Yokohama 236-0027, Japan.}
\address[ya]{Faculty of Science, Yamagata University, Yamagata 990-8560, Japan.}
\address[a]{Department of Mathematics, Ajou University, Suwon 16499, Republic of Korea.}

\begin{abstract}
Let $G$ be a graph, and let $w$ be a positive real-valued weight function on $V(G)$. For every subset $S$ of $V(G)$, let $w(S)=\sum_{v \in S} w(v).$ A non-empty subset $S \subset V(G)$ is a {\it weighted safe set} of $(G,w)$ if, for every component $C$ of the subgraph induced by $S$ and every component $D$ of $G-S$, we have $w(C) \geq w(D)$ whenever there is an edge between $C$ and $D$.
If the subgraph of $G$ induced by a weighted safe set $S$ is connected, then the set $S$ is called a {\it connected weighted safe set} of $(G,w)$. The \textit{weighted safe number} $\s(G,w)$ and \textit{connected weighted safe number} $\cs(G,w)$ of $(G,w)$ are the minimum weights $w(S)$ among all weighted safe sets and all connected weighted safe sets of $(G,w)$, respectively. Note that for every pair $(G,w)$, $\s(G,w) \le \cs(G,w)$ by their definitions. In \cite{FPS:path:cycle},
it was asked which pair $(G,w)$ satisfies the equality and shown that every weighted cycle satisfies the equality.
In this paper, we give a complete list of connected bipartite graphs $G$ such that
$\s(G,w)=\cs(G,w)$ for every weight function $w$ on $V(G)$.
\end{abstract}

\end{frontmatter}

\section{Introduction}
We use \cite{cl} for terminology and notation not defined here.
Only finite, simple (undirected) graphs are considered. For a graph $G$, the subgraph of $G$ induced by a subset
$S \subseteq V(G)$ is denoted by $G[S]$.
We often abuse or identify terminology and notation for
subsets of the vertex set and subgraphs induced by them. In particular, a
component is sometimes treated as a subset of the vertex set.
For a subset $S$ of $V(G)$, we denote $G[V(G)\setminus S]$ by $G-S$.
For a graph $G$, when $A$ and $B$ are disjoint subsets of $V(G)$, the set of edges joining some vertex
of $A$ and some vertex of $B$ is denoted by $E_G(A,B)$.
If $E_G(A,B)\neq\emptyset$, then $A$ and $B$ are said to be \textit{adjacent}.
A (vertex) weight function $w$ on $V(G)$ means a mapping associating each vertex in $V(G)$ with a positive real number.
We call $(G,w)$ a weighted graph.
For every subset $X$ of $V(G)$, let $w(X)=\sum_{v \in X} w(v)$, and we also write $w(X)$ for $w(G[X])$.

Let $G$ be a connected graph.
A non-empty subset $S \subseteq V(G)$ is a \textit{safe set} if,
for every component $C$ of $G[S]$ and every component $D$
of $G-S$,  we have $|C|\ge |D|$ whenever $E_G(C,D)\neq\emptyset$.
If $G[S]$ is connected, then $S$ is called a \textit{connected safe set}.
In \cite{wsf},  those notions are extended to (vertex) weighted graphs.
Let  $w$  be a weight function on $V(G)$.
A non-empty subset $S \subset V(G)$ is a {\it weighted safe set} of $(G,w)$ if, for every component $C$ of $G[S]$ and every component $D$ of $G-S$, we have
$w(C) \geq w(D)$ whenever $E_G(C,D) \neq \emptyset$.
The \textit{weighted safe number} of $(G,w)$ is the minimum weight $w(S)$ among all weighted safe sets of $(G,w)$, that is,
$$\s(G,w)=\min\{ w(S) \mid S \text{ is a weighted safe set  of }(G,w)\}.$$
If $S$ is a weighted safe set of $(G,w)$ and $w(S)=\s(G,w)$, then $S$ is called a {\it minimum weighted safe set\/}.
Similar to connected safe sets, if  $S$ is a weighted safe set of $(G,w)$ and $G[S]$ is connected, then $S$ is called a \textit{connected weighted safe set} of $(G,w)$.
The \textit{connected weighted safe number} of $(G,w)$ is defined by
$$\cs(G,w)=\min\{ w(S) \mid S \text{ is a connected weighted safe set  of }(G,w)\},$$
and  a {\it minimum connected weighted safe set\/} is a connected weighted safe set $S$ of $(G,w)$ such that $w(S)=(G,w)$.
It is easy to see that for every weighted graph $(G,w)$,  $\s(G,w) \le \cs(G,w)$ by their definitions.
Throughout this paper, we often abbreviate `weighted' to a weighted safe set or a connected weighted safe set when it is clear from the context.

The notion of a safe set was originally introduced by Fujita et al. \cite{sf} as a variation of facility location problems. A lot of work has been done in this topic. For example, Kang et al. \cite{kkp} explored the safe number of the Cartesian product of two complete graphs, and Fujita and Furuya \cite{SF2018} studied the relationship between the safe number and the integrity of a graph. For a real application, the weighted version of this notion was proposed by Bapat et al. \cite{wsf}. Let $(G, w)$ be a weighted graph.
We can regard $(G, w)$ as a kind of network with certain properties. As discussed in \cite{wsf}, the concept of a safe set can be thought of a suitable measure of network majority and network vulnerability.

In view of such applications, weighted safe set problems in graphs attract much attention, especially from the algorithmic point of view. Let us briefly look back some known results.
Fujita et al. \cite{sf} showed that computing the connected safe number of $(G, w)$ when $w$ is a constant weight function is NP-hard in general. However,
when $G$ is a tree and $w$ is a constant weight function, they constructed a linear time algorithm for computing the connected safe number of $G$. \'{A}gueda et al. \cite{safe2018} gave an efficient algorithm for computing the safe number of an unweighted graph with bounded treewidth. Bapat et al. \cite{wsf} showed that computing the connected weighted safe number in a tree is NP-hard even if the underlyining tree is restricted to be a star. They also constructed an efficient algorithm computing the safe number for a weighted path. Furthermore, Fujita et al. \cite{FPS:path:cycle} constructed a linear time algorithm computing the safe number for a weighted cycle.
Ehard and Rautenbach \cite{ehard} gave a polynomial-time approximation scheme (PTAS) for the connected safe number of a weighted tree.
The parameterized complexity of safe set problems was investigated by Belmonte et al. \cite{b}.

In contrast with the above algorithmic approaches, in this paper, we are concerned with a more combinatorial aspect on weighted safe set problems.
Namely, we would like to find graphs $G$ with a stable structure such that $\s(G,w)=\cs(G,w)$ holds for any choice of the weight function $w$ on $V(G)$.
From the inequality $\s(G,w) \le \cs(G,w)$,
it would be natural to ask which pair $(G,w)$ satisfies the equality.
In this paper, we focus on a much stronger property: Namely, we would like to characterize a graph $G$ such that $\s(G,w)=\cs(G,w)$ not only for a fixed $w$, but also for any arbitrary choice of $w$. As a purely combinatorial problem, it would be interesting to investigate the structure in such special graphs.

Returning to the application aspect on safe set problems, let us recall that the notion of safe sets in graphs was invented for finding a safe place in some graph network model. If the minimum safe place has a connected structure, then
it would definitely be convenient for the refugees to communicate with each other on the safe place. Note that, in the weighted case, one can regard the weight on a vertex as the capacity of the number of people to stay there. From this point of view, we can say that a graph $G$ has a stable structure if $\s(G,w)=\cs(G,w)$ holds for any choice of the weight function $w$ on $V(G)$. For convenience, let us define $\G$ by the family of all graphs $G$ such that $\s(G,w)=\cs(G,w)$ holds for every weight function $w$ on $V(G)$.

As a related work, we found a common property in terms of the weighted safe number sometimes yields a characterization of graphs.
Indeed, Fujita et al. \cite{FPS:path:cycle} showed that a graph $G$ is a cycle or a complete graph if and only if $\s(G,w)\ge w(G)/2$ for every weight function $w$ on $V(G)$.
In the process of this work, the authors already proposed our main problem as the following open problem.

\begin{problem}[\cite{FPS:path:cycle}]\label{problem2:graphs}
Determine  the family of graphs $\G$.
\end{problem}

By definition, when we check whether a graph $G$ belongs to $\G$ or not, we must look at $(G, w)$ in all possible weights yielded by $w$, meaning that we must always deal with infinite cases of $w$. Naturally, it would be a difficult question to ask whether $G\in \G$ or not for a given graph $G$.
However, if we could have a complete answer to Problem~\ref{problem2:graphs}, then it would contribute to the real applications such as network majority and network vulnerability. This is because, the invariable property from any choice of $w$ as defined in $\G$ often plays an important role in stable networks.
We also remark that, as demonstrated in \cite{FPS:path:cycle}, some consideration on paths and cycles in view of $\G$ provides a nice  observation on a problem in combinatorial number theory to find some special partitions of number sequences (see  \cite{FPS:path:cycle} for details). Thus, our problem is important in both theoretical and practical directions.

Unfortunately we could not give the complete answer to Problem~\ref{problem2:graphs}. Yet we achieved a substantial progress on this problem. To state this, we start with the following observation on $\G$.

It is clear that a complete graph is in $\G$. In \cite{FPS:path:cycle}, it was shown that a graph $G$ with $\Delta(G)=|V(G)|-1$ belongs to $\G$  and the following theorem was obtained.

\begin{thm}[\cite{FPS:path:cycle}]\label{thm:cycle}
A cycle belongs to $\G$.
\end{thm}

In this paper, we completely characterize all chordal graphs and all bipartite graphs in $\G$.
A \textit{dominating clique} is a dominating set which is a clique, that is, it induces a complete graph and every vertex $v$ not in the clique has a neighbor in the clique.
\begin{thm}\label{thm:chordal}
Let $G$ be a connected chordal graph. The following are equivalent:
\begin{itemize}
\item[\rm(i)] $G$ has a dominating clique;
\item[\rm(ii)] $\diam(G)\le 3$;
\item[\rm(iii)] $G\in \G$.
\end{itemize}
\end{thm}

In addition, we show that a triangle-free graph in $\G$ has small diameter.

\begin{thm}\label{main:thm:triangle-free}
If $G$ is a triangle-free connected graph in $\G$ which is not a cycle, then $\mathrm{diam}(G)\le 3$.
\end{thm}

The following,  the main result of the paper, gives the complete list of the connected bipartite graphs in $\G$. A \textit{double star} is a tree with diameter at most three.

\begin{defn}\label{def:graph:D:D*}
Let $m$, $n$, $p$, $q$ be nonnegative integers.
Let $D({m,n};p,q)$ (resp. $D^*({m,n};p,q)$) be a connected bipartite graph with bipartition $(X_1\cup X_2\cup P, Y_1\cup Y_2\cup Q)$, where the unions are disjoint,
satisfying  {\rm(1)}-{\rm(4)}:
\begin{itemize}
\item[\rm(1)] $|X_1|=m$,  $|Y_1|=m+1$, $|X_2|=n+1$, $|Y_2|=n$, $|P|=p$, and $|Q|=q$;
\item[\rm(2)] Both $G[X_1\cup Y_1]$ and  $G[X_2\cup Y_2]$ are  complete bipartite graphs;
\item[\rm(3)] The vertices in $P$ are pendant vertices which are adjacent to a vertex $y\in Y_1$
and the vertices in $Q$ are pendant vertices which are adjacent to a vertex $x\in X_2$;
\item[\rm(4)] $E_G(X_1,Y_2)=\emptyset$ and $G[X_2\cup Y_1]$ is a complete bipartite graph (resp. a double star with a dominating edge $xy$).
\end{itemize}
Note that each of $D({m,n};p,q)$ and $D^*({m,n};p,q)$ has a dominating edge $xy$ ($x\in X_2$ and  $y\in Y_1$), where a \textit{dominating edge} is a dominating clique of order two.
See Figure~\ref{fig:graphs} for examples.
\end{defn}
\begin{figure}[h!]
\centering
\includegraphics[width=12cm,page=1]{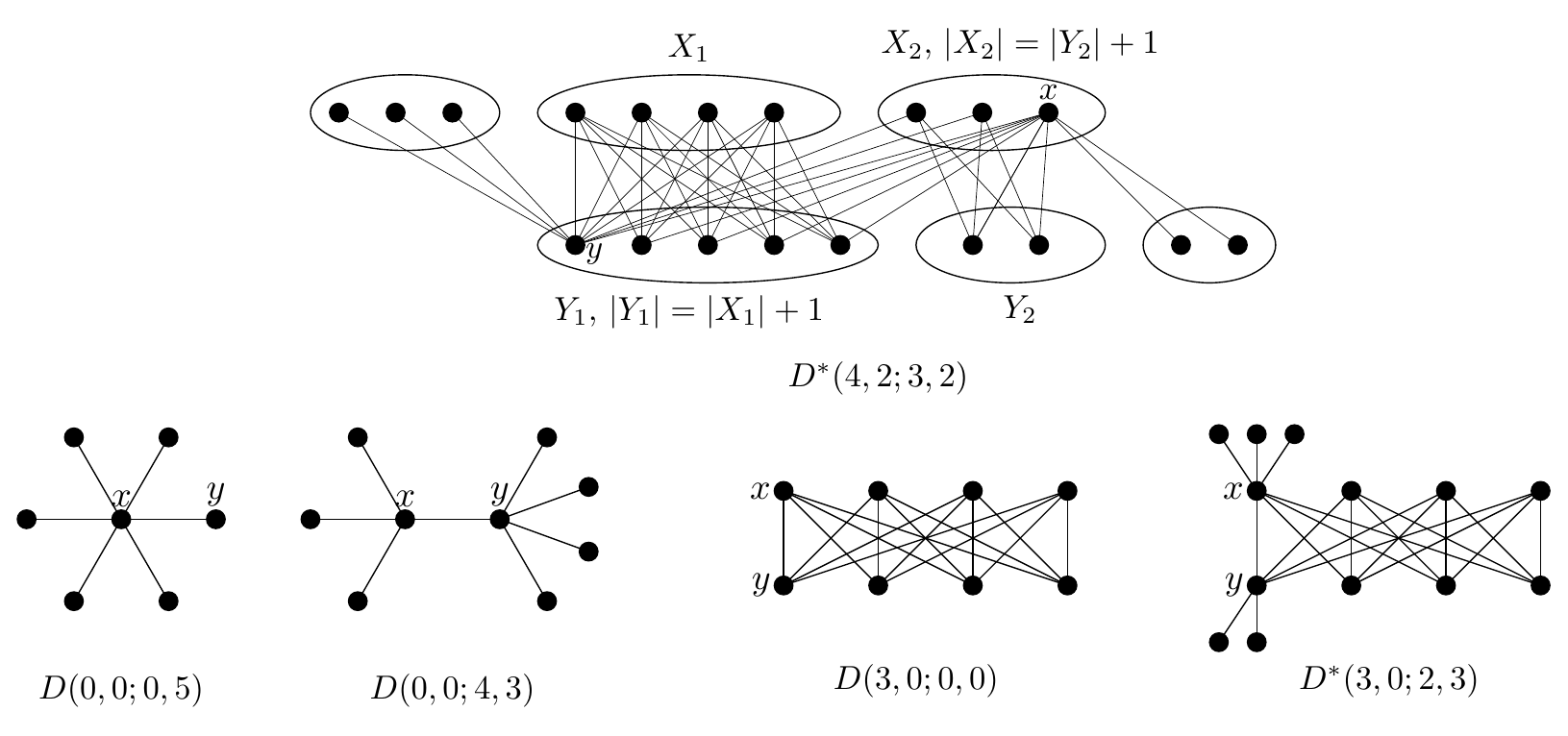}\\
\caption{Examples of graphs $D(m,n;p,q)$ or $D^*(m,n;p,q)$.}\label{fig:graphs}
\end{figure}

The $m$-\textit{book graph}, denoted by $B_m$, is  the Cartesian product of a star $K_{1,m}$ and a path $P_2$. See Figure~\ref{fig:Book}.
The following is our main theorem, which gives a full list of graphs in $\G$ for the bipartite case.

\begin{figure}[h!]
\centering
\includegraphics[width=6cm,page=2]{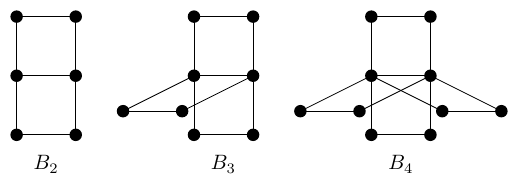}\\
\caption{Examples of book graphs.}\label{fig:Book}
\end{figure}

\begin{thm}[Main Theorem]\label{main:thm:bipartite:CS}
A connected bipartite graph $G$ belongs to $\G$ if and only if $G$ is one of the following:

\begin{itemize}
\item[\rm(I)] an even cycle $C_{2n}$ with $n\ge 2$;
\item[\rm(II)] a double star;
\item[\rm(III)] a book graph $B_n$ with $n\ge 1$;
\item[\rm(IV)] a graph obtained from $K_{3,3}$ by deleting an edge;
\item[\rm(V)]$D(m,n;p,q)$ or $D^*(m,n;p,q)$, with $m\ge 2$, $n\neq 1$ and $p,q\ge 0$,
\end{itemize}
\end{thm}

From our main theorem, we see that if a bipartite graph $G$ belongs to $\G$, then $G$ is an even cycle or $G$ has a dominating edge.
When considering a safe set $S$ of a graph $G$, note that we always observe the bipartite structure between $G[S]$ and $G-S$.
From this view point, we believe that our main theorem settles an essential case of Problem~\ref{problem2:graphs}, which is very far from trivial to prove.

In fact we prepare a companion paper \cite{isaac} in which we show that, for any graph $G$ in the list of Theorem~\ref{main:thm:bipartite:CS} and
for any non-negative weight function $w$ of $G$, there exists a fully polynomial-time approximation scheme (FPTAS) for computing a minimum connected
safe set of $(G,w)$, and moreover, we give a linear time algorithm to decide whether a graph is in the list of Theorem~\ref{main:thm:bipartite:CS} or not.
As byproduct of the above results, it is also shown in \cite{isaac} that there exists an FPTAS for computing a minimum connected safe set of a weighted tree.
This made a substantial progress on the relevant work due to Ehard and Rautenbach \cite{ehard}.

This paper is organized as follows. Section~\ref{sec:Prel} gives preliminaries.
Section~\ref{sec:contraction} provides some lemmas concerning the graphs not in $\G$ in view of a contraction argument, which are useful to prove our main results in the subsequent  sections.
Section~\ref{sec:CS:candidates} finds some graphs in $\G$ with a dominating clique, especially focusing on chordal graphs and bipartite graphs.
This section also provides the proof of Theorem~\ref{thm:chordal}. Finally, Section~\ref{sec:main} provides the proofs of our main results, Theorems~\ref{main:thm:triangle-free} and~\ref{main:thm:bipartite:CS}.

\section{Preliminaries}\label{sec:Prel}

For a connected graph $G$ and $S\subset V(G)$, we denote  by $\beta(G,S)$  the graph whose vertices are the components of $G[S]$ and of $G-S$, and two vertices of $A$ and $B$ are adjacent in $\beta(G,S)$ if and only if $E_G(A,B)\neq \emptyset$ (Figure~\ref{fig:beta}). Note that $\beta(G,S)$ is always a bipartite graph.

\begin{figure}[h!]
\centering
\includegraphics[width=11cm,page=3]{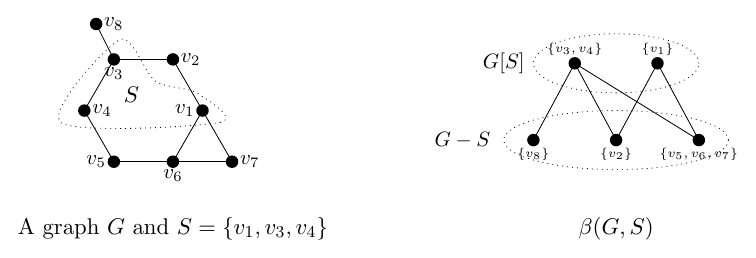}\\
        \caption{An example of $\beta(G,S)$.}\label{fig:beta}
\end{figure}

\begin{lem}\label{lem:intro:beta}
Let $G$ be a connected graph not in $\G$.
If $S$ is a minimum safe set of $(G,w)$ for some weight function $w$ on $V(G)$ such that $\s(G,w)<\cs(G,w)$,
then $\beta(G,S)\not\in \G$.
\end{lem}
\begin{proof}
Let $\beta:=\beta(G,S)$ and let $w_{\beta}$ be a weight function on $V(\beta)$ defined by $w_\beta(D):=w(D)$ for each vertex $D$ of $\beta$. Let $\mathcal{S}$ be the set of the components of $G[S]$.
Then clearly, $\mathcal{S}$ is a safe set of $(\beta,w_\beta)$, and therefore,  $\s(\beta,w_\beta)\le w_\beta(\mathcal{S})=w(S)$.
Suppose that  $\s(\beta,w_\beta)=\cs(\beta,w_\beta)$. Then there is a connected safe set $\mathcal{S}'$ of $(\beta,w_\beta)$ such that
$\cs(\beta,w_\beta)=w_\beta(\mathcal{S}')$, which implies that $S'=\cup_{D\in \mathcal{S}'} D$ is a connected safe set of $(G,w)$. In addition,  $w(S')=w_\beta(\mathcal{S}')=\cs(\beta,w_\beta)=\s(\beta,w_\beta) \le w(S)$, which is a contradiction.
Hence, $\s(\beta,w_\beta)\neq \cs(\beta,w_\beta)$ and so $\beta$ does not belong to $\G$.
\end{proof}

The following proposition is a direct consequence of Lemma~\ref{lem:intro:beta}.
\begin{prop}\label{rmk:intro:contraction}
For a connected graph $G$, if $\beta(G,S)\in \G$ for every $S\subset V(G)$, then $G$ belongs to $\G$.
\end{prop}

Now we give the following observation.

\begin{lem}\label{easy:observation}
Let $G$ be a connected graph such that $\s(G,w)<\cs(G,w)$ for some weight function $w$, and $S$ be a minimum safe set of $(G,w)$.
Then $G-S$ is disconnected.
\end{lem}

\begin{proof}
Let $D_1,\ldots, D_k$ be
the components of $G[S]$.
Note that $k\ge 2$, since $\s(G,w) <\cs (G,w)$.
Without loss of generality, we may assume that
$w(D_1)=\min\{w(D_1), \ldots, w(D_k)\}$.
If $G-S$ is connected, then $V(G) \setminus D_1$ is a connected
safe set of $(G,w)$ whose weight is at most $w(S)$, which is a
contradiction.
\end{proof}

The graph family $\G$ is not changed even if we allow a weight function to include $0$. A nonnegative weight function on $V(G)$ means a mapping associating each vertex with a nonnegative real number, and note that the notions of $\s(G,w)$ and $\cs(G,w)$ are well-defined for a graph $G$ and  a nonnegative weight function $w$ on $V(G)$.
Let $\G_0$ be the set of graphs $G$ such that $\s(G,w)=\cs(G,w)$ for every nonnegative weight function $w$ on $V(G)$.
As the family $\G$ is defined by assuming that the values of all weight functions have positive real numbers,  it is clear that $\G_0\subset \G$. In fact, the equality holds.
\begin{prop}\label{prop:=<}
It holds that $\G= \G_0 $.
\end{prop}

\begin{proof}
Suppose that $\G\neq  \G_0 $.
 Since $\G_0\subset \G$, there exist a connected graph $G\in \G\setminus \G_0$ and  a nonnegative weight function
$w$ on $V(G)$ such that
$\s(G,w) < \cs(G,w)$ and
$\{x \in V(G) : w(x)=0\}\neq \emptyset$.
For simplicity, let $Z=\{x \in V(G) : w(x)=0\}$.

Suppose that $U$ is a minimum safe set of $(G,w)$, and so $w(U)=\s(G,w)$.
For every positive real number $\epsilon$, let us define a positive
weight function $w_{\epsilon}$ on $V(G)$, as follows:
\[w_{\epsilon}(x):=
\begin{cases}
w(x)+\epsilon|Z|& \text{if }x \in  U,\\
\epsilon & \text{if }x \in Z\setminus U ,\\
w(x)&\text{otherwise.} \\
\end{cases}\]
Then for every component $D$ of $G[U]$ and
every component $T$ of $G-U$ such that $E_G(T,D)\neq \emptyset$,
it holds that
\[w_{\epsilon}(T) \le  w(T) +\epsilon|Z| \le w(D)+\epsilon|Z||D|=w_{\epsilon}(D),\] which implies that
$U$ is  a safe set of $(G,w_{\epsilon})$. Thus
\begin{eqnarray}\label{eq:nonnegative}
\forall \epsilon>0, \quad\s(G,w_{\epsilon})\le w_{\epsilon}(U).\end{eqnarray}
In addition,
$w_{\epsilon}(U)=w(U)+\epsilon |Z||U|$.
Thus, together with the fact that $w(U)=\s(G,w)$,
\[w_{\epsilon}(U) =  \s(G,w)+\epsilon |Z||U|.\]
Let  $\epsilon_1$ be a  positive real number so that $\s(G,w)+\epsilon_1 |Z||U|< \cs(G,w)$.
Let $\epsilon_{i+1}=\epsilon_i/2$ for $i\ge 1$.
Then
\begin{eqnarray}\label{eq:nonnegative2}
\forall i\ge 1, \quad w_{\epsilon_i}(U)< \cs(G,w). \end{eqnarray}
For every integer $i\ge 1$, let $S_i$ be a minimum connected safe set of $(G, w_{\epsilon_i})$. Then $w_{\epsilon_i}(S_i)=\cs(G,w_{\epsilon_i})=\s(G,w_{\epsilon_i})$ by the assumption that $G\in \G$ and the fact that $w_{\epsilon_i}$ is a positive weight function.
Together with \eqref{eq:nonnegative} and \eqref{eq:nonnegative2},  \[ w_{\epsilon_i}(S_i)=\s(G,w_{\epsilon_i})\le w_{\epsilon_i}(U)< \cs (G,w),\]
which implies that $S_i$ cannot be a connected safe set of $(G,w)$.

Since $V(G)$ is finite, there exists $S\subset V(G)$ such that  $S$ appears infinitely many times in the sequence $\{S_i\}_{i=1}^{\infty}$.
Then there exists an integer-valued function $\sigma$ such that
$S=S_{\sigma(i)}$ and $\sigma(i) < \sigma(i+1)$ for every positive integer $i$.

Clearly, $S=S_i$ for some $i$, and so $G[S]$ is connected.
Note that  for every positive integer $i$, for every component $T$ of $G-S$,  $w_{\epsilon_{\sigma(i)}}(T) \leq w_{\epsilon_{\sigma(i)}}(S)$.
Since the sequence
$\{ w_{\epsilon_{\sigma(i)}} \}_{i=1}^{\infty}$ converges (uniformly) to $w$,
it holds that $w(T) \le w(S)$  for every component $T$ of $G-S$, which implies that $S$ is a connected safe set of $(G,w)$.
We reach a contradiction to the observation that $S_i$ cannot be a connected safe set of $(G,w)$.
\end{proof}

Thanks to Proposition~\ref{prop:=<}, we allow a nonnegative weight function on the vertex set of a graph when we determine whether a given graph belongs to $\G$ or not.

\begin{prop}\label{prop:pendant}
Let $G$ be a connected graph in $\G$ with a cut vertex $v$.
For each component $D$ of $G-v$, $G[D\cup\{v\}]$ is in $\G$.
\end{prop}

\begin{proof}
Let $H=G[D\cup\{v\}]$ where $D$ is a component of $G-v$, and
suppose that $H$ is not in $\G$. Then there is a weight function $w^H$ on $V(H)$ such that $\s(H,w^H)<\cs(H,w^H)$.
We define a vertex weight function $w$ on $V(G)$ by
\[w(x)=\begin{cases}
  w^H(x) & \text{if }x\in V(H)\\
  0 & \text{otherwise}.
\end{cases}\]
Let $S$ be a minimum safe set of $(H,w^H(x))$.
By the definition of $w$, it is easy to observe that $S$ is also a safe set of $(G,w)$ and so $\s(G,w)\le w(S)=w^H(S)=\s(H,w^H)$, and thus,
\begin{eqnarray}\label{eq1:pendant}
&& \s(G,w)\le \s(H,w^H).
\end{eqnarray}
Now we take a minimum connected safe set $U$ of $(G,w)$, i.e., $w(U)=\cs(G,w)$.
Let $U^H=U\setminus(V(G)\setminus V(H))$.
From the facts that $w(x)=0$ for all vertices $x$ in $V(G)\setminus V(H)$ and $v$ is a cut vertex, $U^H$ is a connected safe set of $(H,w^H)$ and so $\cs(H,w^H)\le w^H(U^H)=w(U)=\cs(G,w)$.
Thus  $\cs(H,w^H)\le \cs(G,w)$.
Hence, together with~\eqref{eq1:pendant},
\[  \s(G,w)\le \s(H,w^H) < \cs(H,w^H)\le \cs(G,w),\]
which implies that $G\not\in \G$, a contradiction.
\end{proof}

From Proposition~\ref{prop:pendant}, it follows that  for a connected graph $G$ in $\G$, each block of $G$ is in $\G$.
Hence, if we add a pendant edge to a graph not in $\G$, then the resulting graph is also not in $\G$.

\section{Contractions and the graphs not in $\G$}\label{sec:contraction}

A graph $G$ is \textit{contractible} to a graph $H$ (or $H$-\textit{contractible}) if $H$ can be obtained from
a partition of $V(G)$ by contracting each part to a vertex. Equivalently, a graph $G$ is \textit{contractible} to $H$ if there is a surjection $\phi : V (G) \rightarrow V (H)$ satisfying the following:
\begin{itemize}
\item[] $E_G( \phi^{-1}(h_i), \phi^{-1}(h_j))\neq \emptyset$ if and only if $h_ih_j \in E(H)$, for every $h_i,h_j \in V (H)$.
\end{itemize}
For each $h \in V (H)$, $\phi^{-1}(h)$ is called a \textit{bag}.
A bag is said to be \textit{connected} if it induces a connected graph in $G$.

In this section, we present several sufficient conditions for a graph not to be in $\G$ in terms of the above contraction argument. The lemmas obtained in this section play an important role in proving our main results.

\subsection{Graphs contractible to a subgraph of $K_{2,3}$}

In this subsection, we discuss some graphs contractible to some subgraphs of $K_{2,3}$. More precisely, we consider $H_i$-contractible graphs where $H_i$ are the graphs as described in Figure~\ref{fig:h1h2h3}, such that the bags corresponding to $u_2$ and $u_4$ are always connected.

We remark that for each $i\in\{1,2,3\}$, $H_i$ does not belong to $\G$.
Here $\alpha$ is a real number such that $\alpha>1$ and let $w_i$ be a weight function on $V(H_i)$ defined by $w_i(u_4)=w_i(u_5)=\alpha$ and $w_i(u_2)=\alpha+1$.
If $i=1$, then $w_1(u_1)=\alpha+1$ and $w_1(u_3)=1$.
If $i\neq 1$, then $w_i(u_1)=w_i(u_3)=\alpha$.
Then for each $i\in\{1,2,3\}$, $\{u_2,u_4\}$ is a unique minimum safe set  of $(H_i,w_i)$ and therefore $\s(H_i,w_i)<\cs(H_i,w_i)$.

\begin{figure}[h!]
\centering
\includegraphics[width=12cm,page=4]{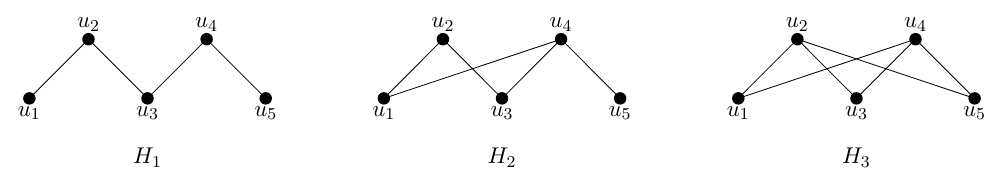}\\
    \caption{All subgraphs of $K_{2,3}$ not in $\G$.
 }\label{fig:h1h2h3}
\end{figure}

Here are several assumptions and common notation throughout this subsection (in Lemmas~\ref{prop:contractible:path} -~\ref{prop:contractible:K23}). We assume that $G$ is a connected graph which is contractible to $H_i$  for some $i\in \{1,2,3\}$, and let $V_j$ be the bag corresponding to $u_j$ of $H_i$ for each $j\in\{1,\ldots,5\}$. In addition, we assume that $V_2$ and $V_4$ are connected bags and let $\alpha>1$ be a sufficiently large real number.

\begin{lem}\label{prop:contractible:path}
Let $G$ be a connected graph contractible to $H_1$. Then $G\not\in\G$.
\end{lem}

\begin{proof}
Take $v_j\in V_j$ for each $j\in\{1,\ldots,5\}$ so that $v_1v_2\in E(G)$ and $v_4v_5\in E(G)$.
We define a weight function $w$ on $V(G)$ such that
\[ w(x)=\begin{cases}
  \alpha+1 &\text{if }x\in\{v_1,v_2\},\\
  \alpha  &\text{if }x\in\{v_4,v_5\},\\
  1/|V_3|  &\text{if }x\in V_3,\\
 0 &\text{otherwise.}
\end{cases}\]
Then $V_2\cup V_4$ is a safe set of $(G,w)$ with $w(V_2\cup V_4) =2\alpha+1$.
Suppose that $G\in\G$.
Then there is a connected safe set $S$ of $(G,w)$ with weight at most $2\alpha+1$.
If $\{v_1,v_2\}\subset S$ then $w(S)\ge 2\alpha+2>2\alpha+1$, which is a contradiction.
If $\{v_1,v_2\}\cap S=\emptyset$, then $G-S$ has a component of weight at least $2\alpha+2$, which is a contradiction to the definition of a safe set.
Thus, $|\{v_1,v_2\}\cap S|=1$, and therefore  $w(S\cap (V_1\cup V_2\cup V_3)) \le \alpha+2$.
If $\{v_4,v_5\}\cap S=\emptyset$, then $G-S$ has a component of weight at least $2\alpha$ and $w(S)\le \alpha+2$, which is a contradiction to the definition of a safe set.
Hence $\{v_4,v_5\}\cap S\neq \emptyset$.
Since $G[S]$ is connected, $S\cap V_3\neq \emptyset$, and therefore
\[w(S)\ge w(S\cap\{v_1,v_2\}) + w(S\cap\{v_4,v_5\}) + w(S\cap V_3) \ge 2\alpha+1 + w(S\cap V_3) >2\alpha+1,\]
a contradiction.
\end{proof}

\begin{lem}\label{prop:contractible:C4+}
Let $G$ be a connected graph contractible to $H_2$.
If
$|E_G(V_1,V_2)|=|E_G(V_2,V_3)|=1$, then $G\not\in\G$.
\end{lem}

\begin{proof}
Take $v_j\in V_i$ for each $j\in\{1,\ldots,5\}$ so that $v_1v_2,v_4v_5\in E(G)$ and $v_3$ has a neighbor in $V_2$.
Let $\epsilon$ be a sufficiently small positive real number.
We define a weight function $w$ on $V(G)$ so that
\[ w(x)=\begin{cases}
\alpha &\text{if }x\in\{v_1,v_3,v_5\},\\
\alpha+1 &\text{if }x=v_2,\\
\alpha-(1+\epsilon)(|V_4|-1) &\text{if }x=v_4,\\
1+\epsilon&\text{if }x\in V_4\setminus\{v_4\},\\
0&\text{otherwise.}
\end{cases}  \]
Then $V_2\cup V_4$ is a safe set of $(G,w)$ with $w(V_2\cup V_4) =2\alpha+1$.
Suppose that $G\in\G$.
Then there is a connected safe set $S$ of  $(G,w)$ with weight at most $2\alpha+1$.

Since $w(v_1)+w(v_2)+w(v_3)=3\alpha+1$, $|S\cap \{v_1,v_2,v_3\}|\le 2$.
If $S\cap \{v_1,v_2,v_3\}=\emptyset$, then by the assumption that $|E_G(V_1,V_2)|=|E_G(V_2,V_3)|=1$, we have $V_2\cap S=\emptyset$ and so $V_2\cup\{v_1,v_3\}$ is in the same component of $G-S$ whose weight is  $3\alpha+1$, which is a contradiction to the definition of a safe set.
Thus  $1\le |S\cap \{v_1,v_2,v_3\}|\le 2$.

Suppose that $v_2\not\in S$.
If $S\cap \{v_1,v_2,v_3\}=\{v_1,v_3\}$,
then, for $G[S]$ being connected, $S\cap V_4\neq\emptyset$, which implies that
$w(S)\ge w(v_1)+w(v_3)+w(S\cap V_4)>2\alpha +1$, a contradiction. 
Suppose that $S\cap \{v_1,v_2,v_3\}=\{v_1\}$ or $\{v_3\}$. Then  $w(S)\neq 2\alpha+1$ by the definition of the weight function $w$, which implies that $w(S)<2\alpha+1$. On the other hand,
the vertices in $\{v_1,v_2,v_3\}\setminus S$ are in the same component of $G-S$ whose weight is at least $2\alpha+1$, a contradiction to the definition of a safe set.

Now suppose that $v_2\in S$.
If $S\cap V_4 = \emptyset$, then since  $|E_G(V_1,V_2)|=|E_G(V_2,V_3)|=1$,  for some $j\in \{1,3\}$ $V_j\cup V_4 \cup V_5$ is in the same component of $G-S$ whose weight is $3\alpha$, a contradiction to the definition of a safe set.
Thus  $S\cap V_4\neq \emptyset$.
Then  for  $G[S]$ being connected, $S$ contains $v_1$ or $v_3$, which implies that $w(S)\ge 2\alpha+1+w(V_4\cap S)\ge 2\alpha+1+(1+\epsilon)$, a contradiction.
\end{proof}

\begin{lem}\label{prop:contractible:K23}
Let $G$ be a connected graph contractible to $H_3$.
Suppose that $|V_1|=|V_2|=1$, $V_3$ is connected, and
there is a vertex $v_4\in V_4$ such that  $E_G(\{v_4\},V_3)\neq\emptyset$ and $E_G(\{v_4\},V_5)=E_G(V_4,V_5)$.
Then  $G\not\in\G$.
\end{lem}

\begin{proof} To reach a contradiction, suppose that $G\in\G$. We have the following claim.

\begin{clm}\label{claim:V5}
There is a component $D$ of $G[V_5]$ such that $E_G(D,V_2)\neq\emptyset$ and $E_G(D,V_4)\neq \emptyset$.
\end{clm}
\begin{proof}
Note that for every component $D$ of $G[V_5]$, either $E_G(D,V_2)\neq\emptyset$ or $E_G(D,V_4)\neq \emptyset$.
Let $U$ be the union of the components $D$ of $G[V_5]$ with $E_G(D,V_2)$
$=\emptyset$.
Then $E_G(U,V_4)\neq\emptyset$ and $E_G(U, V(G)\setminus (V_4\cup U))=\emptyset$.
Similarly, let $W$ be the union of the components $D$ of $G[V_5]$ with $E_G(D,V_4)=\emptyset$. Then $E_G(W,V_2)\neq\emptyset$ and $E_G(W, V(G)\setminus(V_2\cup W))=\emptyset$.
Suppose that $V_5\setminus (U\cup W)=\emptyset$.
Since $G$ is $H_3$-contractible, both $U$ and $W$ are nonempty.
By contracting $V'_j$'s where $V'_1=W$, $V'_2=V_2$, $V_4'=V_4$, $V'_5=U$, and $V'_3=V(G)\setminus (V'_1\cup V'_2\cup V'_4\cup V'_5)$,
$G$ is $H_1$-contractible for the graph $H_1$ in Figure~\ref{fig:h1h2h3}, which implies that $G\not\in \G$ by Lemma~\ref{prop:contractible:path}, a contradiction.
Hence, $V_5\setminus (U\cup W)$ is not empty, and so the claim holds.
\end{proof}

Now we let $V_1=\{v_1\}$ and $V_2=\{v_2\}$.
Take a neighbor $v_3\in V_3$ of $v_4$.
By Claim~\ref{claim:V5}, $G[V_5]$ has a component $D$ such that $E_G(D,V_2)\neq\emptyset$ and $E_G(D,\{v_4\})\neq \emptyset$.
We take a neighbor $v_5$ of $v_4$ from $D$.

Now $\epsilon_3>\epsilon_5>\epsilon_4>0$ are sufficiently small real  numbers so that $\frac{1}{n}>\epsilon_3>2n\epsilon_5> 2n^2 \epsilon_4$, where $n=|V(G)|$.
We define a weight function $w$ on $V(G)$ as follows:
\[w(a)=\begin{cases}
\alpha  & \text{if }a=v_1 \\
\alpha+1  & \text{if }a=v_2  \\
1+ \epsilon_3 & \text{if }a\in V_3 \setminus \{v_3 \} \\
\epsilon_4 & \text{if }a\in V_4\setminus \{v_4\} \\
\epsilon_5 & \text{if }a\in V_5\setminus \{v_5 \}, \\
\end{cases}\]
and then we determine the weights of $v_3,v_4,v_5$ so that
 $w(V_3)=w(V_4)=w(D)=\alpha$.

Since $G\in \G$ and  $V_2\cup V_4$ is a safe set of $(G,w)$ with $w(V_2\cup V_4) =2\alpha+1$,
there is a connected safe set $S$ of $(G,w)$ with weight at most $2\alpha+1$. For simplicity, let $X=S\cap \{v_1,v_2,v_3,v_4,v_5\}$.
Since $w(S)\le 2\alpha+1$, $|X| \le 2$. Moreover, we have the following claim.

\begin{clm}\label{claim:contractible:H_3}
We have $|X|=2$, and the following hold.
\begin{itemize}
\item[\rm(1)]  $|X\cap \{v_1,v_2\}|=1$ and $|X\cap \{v_3,v_4,v_5\}|=1$.
\item[\rm(2)]  $|X\cap\{v_2,v_4\}|=1$ and  $|X\cap\{v_1,v_3,v_5\}|=1$.
\end{itemize}
\end{clm}

\begin{proof}
To show (1), suppose that $\{v_3,v_4,v_5\}\cap X= \emptyset$.
Then $\{v_3,v_4,v_5\}$ is contained in a component of $G-S$, which is a contradiction to the definition of a safe set, since (note that $\alpha$ is sufficiently large.)
\begin{eqnarray*}
w(v_3)+w(v_4)+w(v_5)&=&3\alpha-(1+\epsilon_3)(|V_3|-1)-\epsilon_4(|V_4|-1)-\epsilon_5(|V_5|-1)\\
&>&2\alpha+1+(\alpha - n(1+\epsilon_3+\epsilon_4+\epsilon_5)) \ge w(S).
\end{eqnarray*}
Suppose that $\{v_1,v_2\}\cap X=\emptyset$.
Since $w(\{v_1,v_2\})=2\alpha+1$, $\{v_1,v_2\}$ is a component of $G-S$.
Hence, at least one vertex of $V_3$, say $z_3$, belongs to $S$.
Moreover, $w(S)=2\alpha+1$.
Since $w(V_3\cup V_4)=2\alpha$, it follows that $w(V_5\cap S)\ge 1$, and therefore the vertex $v_5$ must be in $S$.
Since $S$ is connected, the vertex $v_4$ must be in $S$.
It follows that
\begin{eqnarray*}
w(S)&\ge& w(z_3)+w(v_4)+w(v_5)\ge 1+\epsilon_3+\alpha-(|V_4|-1)\epsilon_4+\alpha-(|V_5|-1)\epsilon_5 \\
&>& 2\alpha+1 + (\epsilon_3-n\epsilon_4-n\epsilon_5) \\
&>& 2\alpha+1,
\end{eqnarray*}
a contradiction, where the last inequality follows from the choice of $\epsilon_3,\epsilon_4$, and $\epsilon_5$. As $|X|\le 2$, (1) holds.
We note that (1) also implies that $|X|=2$.

Now we show (2).
If $X=\{v_2,v_4\}$, then $S$ has at least one vertex in $V_1\cup V_3\cup V_5$ for $G[S]$ being connected, which implies that $w(S)\ge w(v_2)+w(v_4)+\min\{w(x)\mid x\in S\cap(V_1\cup V_3\cup V_5)\}\ge \alpha+1+(\alpha-n\epsilon_4) + \epsilon_5>2\alpha+1$ by the assumption on $\epsilon_4$ and $\epsilon_5$, a contradiction.
Hence $|X\cap\{v_2,v_4\}|\le 1$. Since $|X|= 2$, $X\cap\{v_1,v_3,v_5\}\neq \emptyset$.
By (1), it remains to show that $|X\cap\{v_1,v_3,v_5\}|\neq 2$.
Suppose that $|X\cap\{v_1,v_3,v_5\}|=2$. Then $\{v_2,v_4\} \subset V(G)\setminus S$.
By the assumption that $V_2=\{v_2\}$ and $v_4$ is a unique vertex in $V_4$ that has a neighbor in $V_5$, if $S\cap V_5\neq\emptyset$, then $S\subset V_5$, a contradiction to the assumption that $|X\cap\{v_1,v_3,v_5\}|=2$.
Thus $S\cap V_5=\emptyset$, and so the vertices in $D\cup\{v_2,v_4\}$ are in the same component of $G-S$, with weight more than $2\alpha+1$, a contradiction to the definition of a safe set. Hence $|X\cap\{v_1,v_3,v_5\}|\neq 2$ and so (2) holds.
\end{proof}

By Claim~\ref{claim:contractible:H_3}, $X=\{v_2,v_3\}$, $\{v_2,v_5\}$ or $\{v_1,v_4\}$.
Suppose that $X=\{v_2,v_3\}$. Since $v_4,v_5$ are in the same component of $G-S$  and its weight is at least $2\alpha- n\epsilon_5$, it holds that $2\alpha-n\epsilon_5\le w(S)$.
Let $k$ be the number of vertices $x$ in $V_3\cap S$ such that $w(x)=1+\epsilon_3$ ($|V_3|-1\ge k$).
Then, since every element in $S\setminus (V_2\cup V_3)$ has weight at most $\epsilon_5$,
 \[  2\alpha-n\epsilon_5\le  w(S)< \alpha+1 + (\alpha-(|V_3|-1)(1+\epsilon_3))+k(1+\epsilon_3)+n\epsilon_5.\]
If $k< |V_3|-1$, then $(\alpha-(|V_3|-1)(1+\epsilon_3))+k(1+\epsilon_3) \le \alpha-1-\epsilon_3$ and so
\[ 2\alpha-n\epsilon_5 \le w(S)\le 2\alpha-\epsilon_3+n\epsilon_5,\]
a contradiction since $\epsilon_3>2n\epsilon_5$.
Hence, $k=|V_3|-1$ and so  $w(S)=2\alpha+1$ and moreover, $S= V_2\cup V_3$. Then $V_1\cup V_4\cup V_5$  is the component of $G-S$ whose weight is more than $2\alpha+1$, which is a contradiction to the definition of a safe set.

Suppose that either $X=\{v_2,v_5\}$ or $X=\{v_1,v_4\}$.
Then $w(X) \ge 2\alpha - n\epsilon_4$.
If $S\cap V_3\neq \emptyset$, then together with  the fact that $n\epsilon_4<\epsilon_3$, we have
\[w(S)\ge w(X)+w(S\cap V_3) \ge 2\alpha - n\epsilon_4+(1+\epsilon_3)>2\alpha+1,\]
 a contradiction to the assumption that $w(S)\le 2\alpha+1$.
Hence, $S\cap V_3=\emptyset$. Then in each case, we will reach a contradiction to the definition of a safe set.
If $X=\{v_2,v_5\}$, then by the assumptions on the vertex $v_4$ and the fact that $v_4\not\in S$, we have $S \subset V_2\cup V_5$ and so the vertices in $V_1\cup V_3\cup V_4$ are contained in the same component of $G-S$ whose weight is more than $2\alpha+1$.
If  $X=\{v_1,v_4\}$, then $S\subset V_1\cup V_4\cup (V_5\setminus\{v_5\})$ and so $w(S)\le 2\alpha+n\epsilon_5 < 2\alpha+1$,  but the component containing $V_2\cup V_3$ of $G-S$ has weight at least $2\alpha+1$.
\end{proof}

\subsection{Graphs contractible to $K_{m,n}$}

In this subsection, we add one more observation on a contractible structure of a connected graph not in $\G$.

\begin{lem}\label{prop:contractible:Kmn}
Let $G$  be a connected graph contractible to $K_{m,n}$, where $m\neq n$ and $m,n\ge 2$, such that there is at most one bag $Z$ with $|Z|\ge 2$.
If $Z$ is connected, then $G\not\in\G$.
\end{lem}

\begin{proof}
Let $X$ and $Y$ be the partite sets of $K_{m,n}$ such that $X=\{x_1,x_2,\ldots,x_m\}$ and $Y=\{y_1,y_2,\ldots,y_n\}$.
If there is no bag $Z$ with $|Z|\ge 2$, then $G=K_{m,n}$ and then it is easy to show that for a constant weight function $w(x)=1$ , $\s(G,w)=\min\{|X|,|Y|\}<\frac{|X|+|Y|}{2} \le \cs(G,w)$, which implies that $G\not\in \G$.

Now suppose that there is a bag $Z$ with $|Z|\ge 2$.
We may assume that  $x_1\in X$ is the vertex in $K_{m,n}$ corresponding to $Z$.
For simplicity, let $X'=X\setminus \{x_1\}$.
For a sufficiently large real number $\alpha>1$,
a sufficiently small real number $\epsilon$ so that $1>\epsilon(|Z|-1)>0$, and
a fixed vertex $z\in Z$, we define a weight function $w$ on $V(G)$ as follows:
\[w(v)=\begin{cases}
\alpha-\epsilon(|Z|-1) &\text{if }v=z , \\
\epsilon &\text{if }v\in Z\setminus\{z\},\\
\alpha & \text{otherwise.} \\
\end{cases}\]
Then both $X'\cup Z$ and $Y$ are safe sets of $(G,w)$ such that $w(X'\cup Z)=m\alpha$ and $w(Y)=n\alpha$.
Suppose that $G\in\G$.
Then there is a connected safe set $S$ of $(G,w)$ with weight at most $\min\{ m\alpha, n\alpha \}$.

Firstly, suppose that the vertices in $(X'\cup\{z\}\cup Y)\setminus  S$ are in the same component of $G-S$.
Then, by the definition of a safe set,
\[w(S) \ge w( (X'\cup\{z\}\cup Y)\setminus S) \ge w(X'\cup\{z\}\cup Y) - w(S) = \alpha(m+n)-\epsilon(|Z|-1)-w(S),\]
and so
\[ 2w(S)\ge \alpha(m+n)-\epsilon(|Z|-1).\]
Thus, together with the fact that $w(S)\le  \min\{ m\alpha, n\alpha \}$, we have
\[\epsilon(|Z|-1) \ge \alpha(m+n)-2w(S) \ge \alpha(m+n)-2 \min\{ m\alpha, n\alpha \}=\alpha|m-n|>1,\] a contradiction by the choice of $\epsilon$.

Secondly, we consider the case where the vertices in $(X'\cup \{z\}\cup Y)\setminus S$ are not in the same component of $G-S$. Then the following claim holds.
\begin{clm}\label{claim:contractible:Kmn}
It holds that $X'\setminus S\neq\emptyset$ and $Y\setminus S\neq\emptyset$.
\end{clm}
\begin{proof}
Note that $S\neq Y$ and so it is clear that $Y\setminus S\neq\emptyset$.
To show that $X'\setminus S \neq \emptyset$ by contradiction, suppose that $X'\subset S$.
Since $S$ is a connected safe set of $(G,w)$,
$S\cap Y\neq \emptyset$ and so $w(S\cap Y)\ge \alpha$.
Since $m\alpha \le w(X')+w(S\cap Y) =w(S) \le m\alpha$, this implies that
$S=X'\cup\{y_i\}$ for some $y_i\in Y$.
Then, since $V(G)\setminus S=(X'\cup Z\cup Y)\setminus S=(Y\setminus\{y_i\})\cup Z$ and every vertex in $Y$ has a neighbor in $Z$,
$G-S$ has only one component $(Y\setminus\{y_i\})\cup Z$.
Then the vertices in $(X'\cup \{z\}\cup Y)\setminus S$ are in the same component of $G-S$,
a contradiction to the case assumption.
Hence, $X'\not\subset S$  and so $X'\setminus S \neq \emptyset$.
\end{proof}

Since $G[X'\cup Y]$ is isomorphic to $K_{m-1,n}$,  the vertices in $(X'\cup Y)\setminus S$ are in the same component of $G-S$ by Claim~\ref{claim:contractible:Kmn}.
By the definition of a safe set,
\[w(S) \ge w( (X'\cup Y)\setminus S) \ge w(X'\cup Y) -w(S)= \alpha(m+n-1)-\epsilon(|Z|-1)-w(S),\]
which implies that $\epsilon(|Z|-1)\ge \alpha(m+n-1)-2w(S)$.
Thus, since $w(S)\le \min\{ m\alpha, n\alpha \}$,
\[\epsilon(|Z|-1) \ge \alpha(m+n-1)-2 \min\{ m\alpha, n\alpha \}=\alpha|m-n|-\alpha.\]
If $|m-n|\ge 2$, we reach a contradiction by the choice of $\epsilon$.
Hence $|m-n|=1$.

Since $\{z\}$ and  $(X'\cup Y)\setminus S$ belong to different components in $G-S$
(by the case assumption), it implies that $z\not\in S$ and $S\cap Z\neq\emptyset$.
Then $\alpha > w(S\cap Z)>0$.
Since
$\alpha\min\{m,n\}=\min\{m\alpha,n\alpha\} \ge w(S)= w(S\cap (X'\cup Y))+w(S\cap Z) >w(S\cap (X'\cup Y)) = \alpha|S\cap (X'\cup Y)|$,
together with the fact that both $|S\cap (X'\cup Y)|$ and $\min\{m,n\}$ are integers,
it follows that  $\min\{m,n\}-1 \ge |S\cap (X'\cup Y)|$ and so
\[ \alpha(\min\{m,n\}-1 ) \ge w(S\cap (X'\cup Y)).\]
Then
\begin{eqnarray*}
  w((X'\cup Y)\setminus S) &=& w(X'\cup Y)-w(S\cap (X'\cup Y)) \\
  &\ge& \alpha(m-1+n) -\alpha(\min\{m,n\}-1)\\
 &=&\max\{m\alpha,n\alpha\}> \min\{m\alpha,n\alpha\}\ge w(S),
\end{eqnarray*}
a contradiction to the definition of a safe set.
\end{proof}

We finish the section with a corollary, which follows from Lemma~\ref{prop:contractible:Kmn} immediately.

\begin{cor}\label{cor:PropThree:closed_neighborhood}
Suppose that there is a vertex $v$ in a connected graph $G$ such that $\deg_G(v)\ge 3$, $N_G(v)$ is an independent set, every vertex in $N_G(v)$ has degree at least two, and $G-N_G[v]$ is connected. Then  $G\not\in\G$.
\end{cor}

\begin{proof}
Let $\deg_G(v)=d$ and let $Z=V(G)-N_G[v]$.
Since $N_G(v)$ is an independent set and  every vertex in $N_G(v)$ has degree at least two, this implies that $Z\neq \emptyset$.
Then contracting $Z$ into one vertex results in $K_{2,d}$ and $d\ge 3$. By Lemma~\ref{prop:contractible:Kmn}, $G \not\in\G$.
\end{proof}

\section{Dominating cliques and  the graphs in $\G$}\label{sec:CS:candidates}

In this section, we consider some chordal graphs and bipartite graphs in $\G$ having a dominating clique.
We give following observation.

\begin{lem}\label{easy:observation:dominating}
Let $G$ be a connected graph with a dominating clique $K$ such that $\s(G,w)<\cs(G,w)$ for some weight function $w$. For every minimum safe set $S$ of $(G,w)$,
the following hold.
\begin{itemize}
\item[\rm(i)]  Each of the sets $S\setminus K$, $K\setminus S$ and $S\cap K$ is nonempty.
\item[\rm(ii)] Each component of $G[S]$ is adjacent to at least two components in $G-S$.
\end{itemize}
\end{lem}

\begin{proof}
Since we have $s(G, w)<cs(G,w)$, note that $G[S]$ is disconnected.
If $S\subset K$ or $K\subset S$ then by the fact that $K$ is a dominating clique, $G[S]$  is connected, a contradiction. Thus $S\setminus K\neq \emptyset$ and  $K\setminus S \neq \emptyset$.
Arguing similarly, we see that, if $K\subset V(G)\setminus S$, then $G-S$ is connected, a contradiction by Lemma~\ref{easy:observation}.
Thus, $K\cap S\neq \emptyset$, and therefore (i) holds.

Let $D_1$,  $\ldots$, $D_k$ ($k\ge 2$) be the components of $G[S]$, and assume that $D_1$ is
the  component  containing $K\cap S$.
Let $T_1$, $\ldots$, $T_l$ ($l\ge 2$) be the components of $G-S$, and assume that $T_1$ is the  component  containing $K\setminus S$.
Note that each $D_i$ is adjacent to $T_1$ and each $T_j$ is adjacent to $D_1$ by the definition of a dominating clique, and so for each $i$ and $j$,
\begin{eqnarray}\label{eqref:dominating:clique0}
&&  w(T_1)\le w(D_i) \qquad \text{ and } \qquad w(T_j)\le w(D_1) .
\end{eqnarray}
To show (ii) by contradiction, suppose that there is a component $D_i$ of $G[S]$ that is adjacent to only one component of $G-S$. Then $D_i$ is adjacent to only $T_1$ among all $T_j$'s.
Without loss of generality, we may assume that $i=2$. Let $S'=(S\setminus D_2)\cup T_1$.
Since $K\subset S'$, it follows that $G[S']$ is connected.
In addition, by~\eqref{eqref:dominating:clique0}, $w(S')=w(S)-w(D_2)+w(T_1) \le w(S)$.
Note that the components of $G-S'$ are  $T_2$, $\ldots$, $T_l$, and $D_2$. For a component $T_j$ of $G-S'$,   $w(T_j)\le w(D_1) \le w(S')$ by \eqref{eqref:dominating:clique0}.
If $w(D_2)\le w(S')$, then $S'$ is a connected safe set with weight at most $w(S)$, a contradiction. Thus
\begin{eqnarray}\label{eqref:dominating:clique}
&&  w(S')<w(D_2).
\end{eqnarray}
Suppose that $w(D_2)>w(S')+w(T_2)+\cdots +w(T_l)=w(V(G)\setminus D_2)$.
Since $G-D_2$ is connected, then $D_2$ is a connected safe set of $(G,w)$ and $w(D_2)<w(S)$, a contradiction.
Thus $w(D_2)\le w(S')+w(T_2)+\cdots +w(T_l)$.
We take the smallest integer $m$ with $2\le m\le l$ such that $w(D_2)\le w(S')+w(T_2)+\cdots +w(T_m)$.
Let $S''=S'\cup T_2\cup \ldots \cup T_m$.
Note that the components of $G-S''$ are $D_2$ and some $T_{j}$'s, where $j>m$.
Then clearly, by the choice of $m$, $w(D_2) \le w(S'')$.
By \eqref{eqref:dominating:clique0}, $w(T_j)\le w(D_1)\le w(S'')$ for all $j>m$.
Hence, $S''$ is a connected safe set of $(G,w)$.

If $m=2$, then $w(S'')=w(S')+w(T_2)< w(D_2)+w(D_1) \le w(S)$ where the first inequality follows from \eqref{eqref:dominating:clique0} and  \eqref{eqref:dominating:clique}.
If $m\ge 3$, then
\[w(S'')=(w(S')+w(T_2)+\cdots w(T_{m-1})) + w(T_m) \le w(D_2)+w(D_1)\le w(S),\]
where the first inequality follows from the choice of $m$ and \eqref{eqref:dominating:clique0}. Then $S''$ is a connected safe set of $(G,w)$, a contradiction.
\end{proof}

\subsection{Chordal graphs: Proof of Theorem~\ref{thm:chordal}}\label{subsec:chordal}

In this subsection, we show that the existence of a dominating clique in a chordal graph $G$ implies $G\in \G$ and the converse is also true.
The following are two known results on chordal graphs.

\begin{thm}[\cite{chordal}]\label{thm:cite:chordal}
Every connected chordal graph $G$ can be contracted to a path of length $\mathrm{diam}(G)$ so that each bag is connected.
\end{thm}

\begin{thm}[\cite{chordal:dominating}]\label{thm:cite:chordal:dominating}
A connected chordal graph $G$ has a dominating clique if and only if $\diam(G)\le 3$.
\end{thm}

Now we are ready to prove Theorem~\ref{thm:chordal}.

\begin{proof}[Proof of Theorem~\ref{thm:chordal}] By Theorem~\ref{thm:cite:chordal:dominating}, it remains to show that (ii) and (iii) are equivalent.
Suppose that  $\diam(G)\ge 4$.
Then by Theorem~\ref{thm:cite:chordal}, $G$ is contractible to a path of length at least four so that each bag is connected. Let $V_1$, $V_2$, $\ldots$, $V_d$ be the connected bags corresponding to that path, where $d= \diam(G)+1\ge 5$.
By considering the partition with
$V_1$, $V_2$, $V_3$, $V_4$, $V(G)\setminus (V_1\cup \cdots \cup V_4)$, we can see that
$G$ is contractible to a path of length exactly four so that each bag is connected and therefore $G\not\in \G$  by Lemma~\ref{prop:contractible:path}. Thus, (iii) implies (ii).

To show that (ii) implies (iii), suppose that  $\diam(G)\le 3$. By Theorem~\ref{thm:cite:chordal:dominating}, there is a dominating clique $K$ of $G$.
To reach a contradiction, suppose that $G\not\in \G$. Then there is a weight function $w$ on $V(G)$ such that $\s(G,w)<\cs(G,w)$.
Let $S$ be a minimum safe set of $(G,w)$.
By Lemma~\ref{easy:observation:dominating}~(i),
 $S\setminus K \neq \emptyset$, $K\setminus S \neq \emptyset$ and $S\cap K \neq \emptyset$.
Let $D_1$, $D_2$, \ldots, $D_k$ be the components of $G[S]$ and $T_1$, $T_2$, $\ldots$,
$T_l$ be the components of $G-S$.
We assume that $D_1$ contains $S\cap K$ and $T_1$ contains $K\setminus S$.

If $E_G(D_i,T_j)\neq \emptyset$ for some $i,j\ge 2$, then the union of $D_i$, $T_1$, $D_1$, $T_j$ contains a cycle and its shortest cycle  is an induced cycle of length at least four, a contradiction to the fact that $G$ is chordal.
Thus for each $i\in\{2,\ldots,k\}$, $D_i$ is adjacent to only $T_1$ among all $T_j$'s, which is a contradiction to Lemma~\ref{easy:observation:dominating} (ii).
\end{proof}

Together with Theorem~\ref{thm:chordal}, the following corollary holds immediately.

\begin{cor}\label{cor:tree}
For a tree $T$, $T\in \G$ if and only if $T$ is a double star.
\end{cor}

 In view of Lemma~\ref{prop:contractible:path}, it is easy to check that a path $P_n$ is in $\G$ if and only if $n\le 4$.

\subsection{Bipartite graphs}\label{subsec:bipartite}

In this subsection, we shall investigate the structure of bipartite graphs in $\G$.
First, we show that every book graph belongs to $\G$.

\begin{prop}\label{prop:book:graph}
For a positive integer $m$, the $m$-book graph $B_m$ belongs to $\G$.
\end{prop}

\begin{proof}
Let $(X,Y)$ be the bipartition of  $B_m$, and $xy$ be a dominating edge, $x\in X$ and $y\in Y$.
Suppose that $G\not\in\G$.
Then there is a weight function $w$ on $V(G)$ such that $\s(G,w) <\cs (G,w)$.
Take a minimum safe set $S$ of $(G,w)$.
By Lemma~\ref{easy:observation:dominating}~(i), we may assume that
$S\cap\{x,y\}=\{x\}$.
Let $D_x$ be the component of $G[S]$ containing $x$ and let $D_y$ be the
component of $G-S$ containing $y$.
Note that all of the components of $G[S]$ other than $D_x$ are isolated vertices in the set $X\setminus\{x\}$.
Let $S\setminus D_x=\{x_1,\ldots,x_l\}$ and let $\{y_i\}=N_G(x_i)\setminus \{y\}$ for each $i\in\{1,\ldots,l\}$.
Then $\{y_1,\ldots,y_l\}\cap S =\emptyset$.
Let us define $S_x=D_x \cup \{y_1,\ldots,y_l\}$
and $S_y=V(G)\setminus S_x$. Note that $S_y=D_y\cup\{x_1,\ldots,x_l\}$.
Then by the definition of a safe set,
\begin{eqnarray*}
w(S)&=&w(D_x) + w(x_1)+\ldots+w(x_l)\ge w(D_x)+w(y_1)+\ldots +w(y_l)=w(S_x)\\
w(S)&=&w(D_x) + w(x_1)+\ldots+w(x_l)\ge w(D_y)+w(x_1)+\ldots+w(x_l)=w(S_y)
\end{eqnarray*}
and so $w(S) \geq \max\{w(S_x), w(S_y)\}$.
Since both $G[S_x]$ and $G[S_y]$
are connected  and $V(G)$ is a disjoint union of $S_x$ and $S_y$,
at least one of $S_x$ and $S_y$ must be a connected
safe set of $(G,w)$ whose weight is at most $w(S)$, which is a
contradiction.
\end{proof}

In the following, we characterize all graphs $D(m,n;p,q)$ or $D^*(m,n;p,q)$ (see Definition~\ref{def:graph:D:D*}) in $\G$.
%
%

Note that a double star with at least two vertices is $D({0,0};p,q)$ for some $p,q$, and $K_{3,3}$ minus an edge is equal to $D({1,1};0,0)$.
In addition, $D(1,0;0,0)=D^*(1,0;0,0)=C_4=B_1$ and $D^*(1,1;0,0)=B_2$.
Hence, the following proposition shows that the graphs described in (II), (IV), or (V) of Theorem~\ref{main:thm:bipartite:CS} are in $\G$.

\begin{prop}\label{prop:case:(5)}
For nonnegative integers $m$, $n$, $p$ and $q$ with $m\ge n$,
let $G$ be a graph either  $D(m,n;p,q)$ or $D^*({m,n};{p,q})$.
Then $G$ belongs to $\G$ if and only if one of the followings holds:
{\rm(a)} $m,n\ge 2$;  {\rm(b)} $m\neq 1$ and $n=0$;
{\rm(c)} $(m,n;p,q)=(1,1;0,0)$;  {\rm(d)} $(m,n;p,q)=(1,0;0,0)$.
\end{prop}

\begin{proof}
Let ($X_1\cup X_2\cup P, Y_1\cup Y_2\cup Q$) be the bipartition of $G$ (following Definition~\ref{def:graph:D:D*}).
For simplicity, let \[X'=X_1\cup X_2, \quad X=X'\cup P, \quad  Y'=Y_1\cup Y_2,\quad Y=Y'\cup Q, \quad
G'=G[X'\cup Y'].\]
See Figure~\ref{fig:graphs_sec4}.

Suppose that $m$, $n$, $p$, and $q$ satisfy none of (a)-(d).
Then either $(m,n;p,q)=(m,1;p,q)$ for some $m\ge 2$,
or $(m,n;p,q)\in \{ (1,1;p,q), (1,0;p,q)\}$ for some $p,q$ with $p>0$ or $q>0$.
Then, in each case, it is easy to see that $G$ is $H_2$-contractible for the graph $H_2$ in Figure~\ref{fig:h1h2h3} so that the bags are $V_1$, $\ldots$, $V_5$ with $|E_G(V_1,V_2)|=|E_G(V_2,V_3)|=1$, and so $G\not\in \G$ by Lemma~\ref{prop:contractible:C4+}.  More precisely,
if $m=1$ then let $V_2:=X_1$ and if $n=1$ then let $V_2:=Y_2$.
Hence the `only if' part holds.

To show the `if' part by contradiction,
suppose that one of (a)-(d) holds and $G\not\in\G$.
Suppose that we take such $G$ so that (1) $|V(G')|$ is minimum,
and (2) $|V(G)|$ is minimum subject to (1).
Then there is a weight function $w$ on $V(G)$ such that $\s(G,w) <\cs (G,w)$. Take a minimum safe set $S$ of $(G,w)$.
Let $\beta=\beta(G,S)$. Note that $\beta\not \in\G$ by Lemma~\ref{lem:intro:beta}.

\begin{clm}\label{claim:m}
It holds that $m\ge 2$.
\end{clm}
\begin{proof}
If $m=0$, then $G=D(0,0;p,q)=D^*(0,0;p,q)$ is a double star and so $G\in\G$ by Corollary~\ref{cor:tree}, a contradiction.
Now suppose that $m=1$. Then (c) or (d) holds.
Since $D(1,0;0,0)=D^*(1,0;0,0)=C_4\in \G$ by Theorem~\ref{thm:cycle}, it holds that
$(m,n;p,q)\neq (1,0;0,0)$.
Since  $D^*(1,1;0,0)=B_2\in \G$ by Proposition~\ref{prop:book:graph},
$(m,n;p,q) \neq D^*(1,1;0,0)$.
Hence, to prove the claim, it is sufficient to show that $G\neq D(1,1;0,0)$.

Suppose to the contrary that $G=D(1,1;0,0)$.
Note that $G$ is a graph obtained from $K_{3,3}$ by deleting an edge,
and let the vertices of $G$ be labeled as the graph in Figure~\ref{fig:spanning:K33}.
\begin{figure}[h!]
\centering
\includegraphics[width=3cm,page=5]{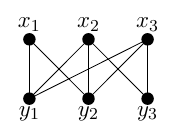}\\
    \caption{$K_{3,3}-e$, where $e=x_1y_3$.}\label{fig:spanning:K33}
    \end{figure}
Since each of $x_2y_1$, $x_2y_2$, $x_3y_1$, $x_3y_2$ is a dominating edge, we may assume that $x_2,x_3\in S$ and $y_1,y_2\not\in S$ by Lemma~\ref{easy:observation:dominating} (i).
If $y_3\in S$, then, for $G[S]$ being disconnected, $x_1\in S$ and so $S=\{x_1,x_2,x_3,y_3\}$, which implies that $\beta(G,S)$ is a cycle of length four and so $\beta\in\G$ by Theorem~\ref{thm:cycle}, a contradiction. Thus $y_3\notin S$. If $x_1\not\in S$, then $\beta(G,S)$ is a cycle of length four, again a contradiction, and therefore $x_1 \in S$.
Hence, $S=\{x_1,x_2,x_3\}$. By the definition of a safe set, $w(x_i)-w(y_i)\ge 0$ for each $i$.
Take $i^*\in\{1,2\}$ such that $w(x_{i^*})-w(y_{i^*})$ is minimum.
Let $S'=(S\setminus \{x_{i^*}\})\cup\{y_{i^*}\}$.
Then $G[S']$ is connected, and moreover,
\begin{eqnarray*}
&&w(S')-w(V(G)\setminus S') =\sum_{i=1}^{3} (w(x_i)-w (y_i))-2(w(x_{i^*})-w(y_{i^*})) \ge 0
\end{eqnarray*}
and so $S'$ is a connected safe set whose weight is not greater than $S$.
Thus  $\s(G,w)=\cs(G,w)$, which is a contradiction.
Hence, $m\ge 2$ and the claim holds.
\end{proof}

Let $xy$ be a dominating edge of $G$ where $x\in X_2$ and $y\in Y_1$.
Note that by Lemma~\ref{easy:observation:dominating} (i), $|S\cap \{x,y\}|=1$. Let $u_x$ and $u_y$ be the vertices of $\beta$ corresponding to the components of $G-S$  or $G[S]$ containing $x$ and $y$, respectively.
Hence, $u_xu_y$ is a dominating edge of $\beta.$

\begin{figure}[h!]
\centering
\includegraphics[width=8cm,page=6]{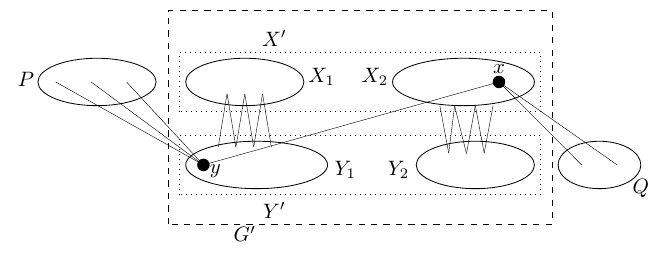}\\
\caption{A graph $G$, where $G[X_2\cup Y_1]$ induces either $K_{n+1,m+1}$ or a double star.}\label{fig:graphs_sec4}
\end{figure}

\begin{clm}\label{claim:beta=G}
It holds that $\beta=G$.
\end{clm}

\begin{proof}
Note that $m\ge 2$ by Claim~\ref{claim:m} and so one of $(a)$ or $(b)$ holds.
Suppose that $\beta\neq G$.
First we claim that some edge of $G'$ is contracted to obtain $\beta$.
If not, then $\beta=D(m,n;p',q')$  for some $p',q'$ with $p'+q'<p+q$ (satisfying the conditions (a) or (b)), which implies that $\beta\in \G$ by the minimality of $|V(G)|$, a contradiction.

Suppose that $(b)$ holds. Since $n=0$, $G'$ is a complete bipartite graph.
Thus $\beta(G',S\cap V(G'))$ is a star by the fact that every edge in $G'$ is a dominating edge of $G'$,
which implies that $\beta$ is a double star. Thus $\beta\in \G$ by Corollary~\ref{cor:tree}, a contradiction.

Suppose that (a) holds.
Without loss of generality we may assume that $x\in S$ and $y\not\in S$ by Lemma~\ref{easy:observation:dominating} (i).
\begin{subclm}\label{subclaim:beta} The following hold:
\begin{itemize}
\item[\rm(i)] Each of $S \cap V(G')$ and $(V(G)\setminus S)\cap V(G')$ induces a disconnected graph in $G$.
\item[\rm(ii)] $X_2\subset S$.
\item[\rm(iii)] $Y_1\cap S=\emptyset$.
\end{itemize}
\end{subclm}
\begin{proof}
If $S \cap V(G')$ or $(V(G)\setminus S)\cap V(G')$ is connected,
then it is easy to check that $\beta$ is a double star with the dominating edge $u_xu_y$, a contradiction. Thus (i) holds.

To show (ii), suppose to the contrary that $X_2\setminus S\neq\emptyset$.
If $X_1\setminus S\neq\emptyset$,
then $(V(G)\setminus S)\cap V(G')$ induces a connected graph, a contradiction to (i). Thus $X_1\subset S$.
Suppose that $Y_1\cap S\neq \emptyset$.
If $Y_2\cap S\neq\emptyset$ or $G=D(m,n;p,q)$, then $S \cap V(G')$ induces a connected graph, a contradiction to (i).
If $Y_2\cap S=\emptyset$ and $G=D^*(m,n;p,q)$, then $\beta$ is a double star with the dominating edge $u_xu_y$, a contradiction.
Thus $Y_1\cap S=\emptyset$.
If $G=D(m,n;p,q)$, then $(V(G)\setminus S)\cap V(G')$ induces a connected graph, a contradiction to (i). Thus, $G=D^*(m,n;p,q)$.
Then $\beta=D(m,0;p',q')$ for some $p',q'\ge 0$ (See Figure~\ref{fig:prop48}). By minimality of $|V(G')|$, $\beta\in\G$, a contradiction.

\begin{figure}[h!]
\centering
  \includegraphics[width=8cm,page=7]{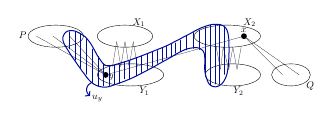}\\
\caption{An illustration for the proof of Subclaim~\ref{subclaim:beta} (ii).}\label{fig:prop48}
\end{figure}

To show (iii), suppose that $Y_1\cap S\neq\emptyset$.
Note that $X_2\subset S$ by (ii).
If $Y_2\cap S\neq \emptyset$ or $G=D(m,n;p,q)$, then $S\cap V(G')$ induces a connected graph, a contradiction to (i). Thus $Y_2\cap S=\emptyset$
and $G=D^*(m,n;p,q)$.
Then $\beta=D(0,n;p',q')$ for some $p',q'\ge 0$. By minimality of $|V(G')|$, $\beta\in\G$, a contradiction.
\end{proof}

We prove that $X'\subset S$ and $Y' \cap S=\emptyset$.
From (ii) and (iii) of Subclaim~\ref{subclaim:beta},
it is sufficient to check $X_1\subset S$ and $Y_2\cap S=\emptyset$.
If $X_1\not\subset S$ and $Y_2\cap S\neq \emptyset$, then
$\beta$ is a double star with the dominating edge $u_xu_y$, a contradiction.
If $X_1\not\subset S$ and $Y_2\cap S=\emptyset$, then
$\beta=D(0,n;p',q')$ for some $p'$ and $q'$, a contradiction to the minimality of $|V(G')|$.
If $X_1\subset S$ and $Y_2\cap S\neq \emptyset$, then
$\beta=D(m,0;p',q')$ for some $p'$ and $q'$, a contradiction to the minimality of $|V(G')|$. Hence, $X'\subset S$ and $Y'\cap S=\emptyset$. This contradicts the observation that some edge of $G'$ must be contracted to obtain $\beta$.
This completes the proof of the claim.
\end{proof}

By Claim~\ref{claim:beta=G}, either $S=X$ or $S=Y$.
Without loss of generality, we may assume that $S=X$.
By Lemma~\ref{easy:observation:dominating} (ii), it follows that $p=0$.
Let $X_1=\{x_1,\ldots, x_m\}$ and $Y_1\setminus\{y\}=\{y_1,\ldots,y_m\}$ (recall that $m,n\ge2$). Without loss of generality, we assume that $w(x_1)\le w(x_i)$ for all $i$.

\medskip

\noindent{\textbf{(Case 1)}} Suppose that $w(y_1)\le w(y)+w(X_2\setminus\{x\})$.
Let $S'=(S\setminus\{x_1\})\cup \{y\}$.
Then $w(S')=w(S)-w(x_1)+w(y)\le w(S)$, since $x_1y\in E(G)$ and $S$ is a safe set.
Moreover, since the dominating edge $xy$ is in $S'$,  $G[S']$ is connected.

Take a component $D$ of $G-S'$.
If $D$ is a singleton, say $D=\{y'\}$, then $y'\in Y_2\cup Q$ and so
$w(y')\le w(x)\le w(S')$, where the first inequality follows from the fact that $S$ is a safe set.
Suppose that $D$ is not a singleton.
Then $D=\{x_1,y_1,\ldots,y_{m}\}$.
Note that $w(x_1)\le w(x_i)$, $w(y_i)\le w(x_j)$, $w(y_i)\le w(x)$ for every $i,j$.
Then by the case assumption,
$w(D)=w(x_1)+w(y_1)+(w(y_2)+\cdots+w(y_{m})) \le
w(x_2)+(w(y)+w(X_2\setminus\{x\})) + (w(x_3)\cdots+w(x_m)+w(x))  = w(S')$.
This implies that $S'$ is a connected safe set of $(G,w)$, a contradiction.

\smallskip

\noindent{\textbf{(Case 2)}} Suppose that $w(y)+w(X_2\setminus\{x\})<w(y_1)$.
Then clearly, we have $G=D^*(m,n;0,q)$.
We take a vertex $x'\in X_2\setminus\{x\}$.
By the case assumption, we have
$w(y)+w(x')< w(y_1)$.
Thus
\begin{eqnarray}\label{eq:lem33}
&&w(x')< w(y_1)-w(y)<w(y_1) \le w(x_1).
\end{eqnarray}
Let $S'=(X\setminus\{x'\})\cup\{y\}$.
Then $w(S')=w(S)-w(x')+w(y)\le w(S)$, since $x'y\in E(G)$ and $S$ is a safe set.
Moreover, since the dominating edge $xy$ is in $S'$,  $G[S']$ is connected.

Take a component $D$ of $G-S'$.
If $D$ is a singleton, say $D=\{y'\}$, then $y'\in Y_1\cup Q$ and so
$w(y')\le w(x)\le w(S')$, where the first inequality is from the fact that $S$ is a safe set.
Suppose that $D$ is not a singleton.
Then $D=\{x'\}\cup Y_2$. By \eqref{eq:lem33} and the fact that there is a perfect matching between $Y_2$ and $X_2\setminus\{x'\}$, we have
\[w(D)=w(x')+w(Y_2) \le
w(x_1)+w(X_2\setminus\{x'\}) \le w(S').\]
This implies that $S'$ is a connected safe set of $(G,w)$, a contradiction.
\end{proof}

\section{Proofs of Theorems~\ref{main:thm:triangle-free} and~\ref{main:thm:bipartite:CS}}\label{sec:main}

\subsection{Proof of Theorem~\ref{main:thm:triangle-free}}

In this subsection,
we often use the lemmas in Section~\ref{sec:contraction}. Throughout the proof, we obtain a partition $\{V_1,\ldots,V_5\}$ of $V(G)$ so that $V_2$ and $V_4$ induce connected graphs (with some additional conditions according to the lemmas), and then we apply those lemmas.

\begin{proof}[Proof of Theorem~\ref{main:thm:triangle-free}]
Suppose to the contrary that there is a triangle-free connected graph $G\in \G$, not a cycle, such that  $\mathrm{diam}(G)\ge 4$.
Let $u$ and $v$ be vertices such that $\mathrm{dist}_G(u,v)=\mathrm{diam}(G)$.
Note that every neighbor of $u$ or $v$ has degree at least two by the maximality of $\mathrm{dist}_G(u,v)$.
For simplicity, let $H_u=G-N_G[u]$ and $H_v=G-N_G[v]$.

\begin{clm} For $a\in\{u,v\}$, if $H_a$ is connected, then $\deg_G(a)\le 2$.
  \label{claim:connected}\end{clm}

\begin{proof}
 Suppose that $H_a$ is connected and $\deg_G(a)\ge 3$.
 Since $G$ is triangle-free, $N_G(a)$ is an independent set.
Moreover, by the maximality of $\mathrm{dist}_G(u,v)$, each neighbor of $a$ has degree at least two.  By Corollary~\ref{cor:PropThree:closed_neighborhood}, $G$ does not belong to $\G$, which is a contradiction.
\end{proof}

\begin{clm}\label{claim:disconnected}
At least one of $H_u$ and $H_v$ is disconnected.
\end{clm}

\begin{proof}
Suppose that $H_u$ and $H_v$ are connected.
By Claim~\ref{claim:connected}, $\deg_G(u)\le 2$ and $\deg_G(v)\le 2$.

\begin{subclm}\label{subclm:triangel-fee}
For $a\in\{u,v\}$, $\deg_G(a)=2$ and the graph $H_a-x$ is disconnected for every vertex $x$ with $\dist_G(a,x)\ge 3$.
\end{subclm}
\begin{proof}
Suppose that $\deg_G(u)=1$ and $\deg_G(v)=1$. Let
$V_1=\{u\}, \quad V_2=N_G(u),\quad  V_4=N_G(v), \quad V_5=\{v\}, \quad  V_3=V(G)\setminus (V_1\cup V_2\cup V_4\cup V_5).$
Note that $V_3\neq \emptyset$ since $\mathrm{diam}(G)\ge 4$. In addition, $G[V_2]$ and $G[V_4]$ are connected. By contracting $V_j$'s, $G$ is $H_1$-contractible for the graph $H_1$ in Figure~\ref{fig:h1h2h3}, and so $G\not\in \G$ by Lemma~\ref{prop:contractible:path}, which is a contradiction.
Without loss of generality, we may suppose that $\deg_G(u)=2$, and let $N_G(u)=\{z_1,z_2\}$.
Suppose that $H_u-x$ is connected for some  vertex $x$ with $\dist_G(u,x)\ge 3$.
Then let
$ V_1=\{z_1\},  V_2=\{u\}, V_3=\{z_2\}, V_5=\{x\},  V_4=V(G)\setminus (V_1\cup V_2\cup V_3\cup V_5)$.
Note that $V_4\neq \emptyset$ since  $V_4=V(H_u-x)$.
Clearly, $V_2$ and $V_4$ induce connected graphs and $|E_G(V_1,V_2)|=|E_G(V_2,V_3)|=1$.
Moreover, since $\dist_G(u,x)\ge 3$, we have $xz_i\not\in E(G)$ for each $i\in\{1,2\}$.
By contracting $V_j$'s, $G$ is $H_2$-contractible for the graph $H_2$ in Figure~\ref{fig:h1h2h3} and so $G\not\in \G$ by Lemma~\ref{prop:contractible:C4+}, which is a contradiction. Thus $H_u-x$ is disconnected for every vertex $x$ with $\dist_G(u,x)\ge 3$.
Then $H_u-v$ is disconnected and so  $\deg_G(v)\neq 1$, which implies that $\deg_G(v)=2$.
By the symmetry of the roles of $u$ and $v$, we can show that $H_v-x$ is disconnected for every vertex $x$ with $\dist_G(v,x)\ge 3$.
\end{proof}

Let  $N_G(u)=\{z_1,z_2\}$ and  $N_G(v)=\{w_1,w_2\}$.
Note that $E_G(\{z_1,z_2\}, \{w_1,w_2\})$
$=\emptyset$ because $\dist_G(u,v)\ge 4$.
We divide the proof into two cases, whether $G-\{u,v\}$ is  connected or not.
See Figures~\ref{fig:proof:triangle-free:Case1}~and~\ref{fig:proof:triangle-free:Case2} for illustrations.

\medskip

\noindent\textbf{(Case 1)} Suppose that  $G-\{u,v\}$ is connected.
Let $D_1$ and $D_2$ be the components of $H_u-v$, and we may assume that $w_i\in D_i$ and $z_i$ has a neighbor in $D_i$. From the case assumption together with the fact that $E_G(\{z_1,z_2\},\{w_1,w_2\})=\emptyset$, we have
  $E_G(\{z_2\},D_1-w_1)\cup E_G(\{z_1\},D_2-w_2)\neq\emptyset$.
Note that if $G[D_i\cup\{z_i\}] - w_i$ is connected for each  $i\in\{1,2\}$, then
$H_v-u$ is connected, which is a contradiction to Subclaim~\ref{subclm:triangel-fee}. Hence we may assume that $G[D_1\cup\{z_1\}] - w_1$ is disconnected.

\begin{figure}[h!]\centering
\includegraphics[width=6cm,page=8]{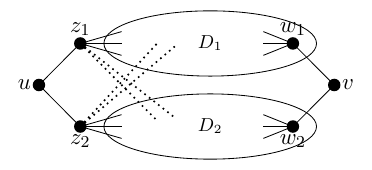}\\
\caption{An illustration for (Case 1) of Claim~\ref{claim:disconnected}. }\label{fig:proof:triangle-free:Case1}
    \end{figure}

Let $D^*_{z_1}$ be the component of $G[D_1\cup\{z_1\}]-w_1$ containing $z_1$, and
let $V_1=\{u\},   \quad  V_2=\{z_2\},  \quad V_3=D_2\cup\{v\}, \quad V_4=D^*_{z_1}\cup\{w_1\},   \quad V_5=V(G)\setminus (V_1\cup V_2\cup V_3\cup V_4)$.
We reach a contradiction by showing that all the conditions of Lemma~\ref{prop:contractible:K23} are satisfied.  Note that each of $V_1$, $V_2$, $V_3$, and $V_4$ induces a connected graph, and since $V_5=D_1\setminus  (D^*_{z_1}\cup\{w_1\})$, $V_5\neq \emptyset$.
In addition, $w_1$ is a unique vertex of $V_4$ having a neighbor in $V_5$ and $w_1$ has a neighbor $v$ in $V_3$.

Now it remains to show that $G$ is $H_{3}$-contractible by contracting $V_j$'s for the graph $H_3$ in Figure~\ref{fig:h1h2h3}.
Since both $H_v-u$ and $H_u-v$ are disconnected by Subclaim~\ref{subclm:triangel-fee}, $E_G(V_2,V_4)=\emptyset$ and $E_G(V_3,V_5)=\emptyset$.
Since $N_G(u)\cap V_3=N_G(u)\cap V_5=\emptyset$, it is clear that $E_G(V_1,V_3)=E_G(V_1,V_5)=\emptyset$.
By the structure, it is clear that each of $E_G(V_1,V_2)$, $E_G(V_1,V_4)$, $E_G(V_3,V_2)$, $E_G(V_3,V_4)$ and $E_G(V_5,V_4)$ is nonempty. To check that $E_G(V_2,V_5)\neq \emptyset$,
take a pendant vertex $x$ of a spanning forest of $G[V_5\cup\{w_1\}]$ other than $w_1$.  Then $H_u-x$ is connected. Hence, $x\in N_G[z_1]\cup N_G[z_2]$. Since $x$ is not in the same component with $z_1$ in $G[D_1\cup\{z_1\}]-w_1$, $x\not\in N_G[z_1]$  and so  $xz_2\in E(G)$. Thus $E_G(V_2,V_5)\neq \emptyset$, and therefore by contracting $V_j$'s, $G$ is
$H_3$-contractible for the graph $H_3$ in Figure~\ref{fig:h1h2h3}.

\medskip

\noindent\textbf{(Case 2)} Suppose that $G-\{u,v\}$ is disconnected.
Note that both  $G-u$ and $G-v$ are connected from the assumption that $\mathrm{dist}_G(u,v)$ is maximum.
Thus, since $\deg_G(u)=\deg_G(v)=2$, $G-\{u,v\}$
 has exactly two components $D_1$ and $D_2$.
Without loss of generality,  let $D_1\supset\{z_1,w_1\}$ and $D_2\supset\{z_2,w_2\}$.

\begin{figure}[h!]\centering
\includegraphics[width=6cm,page=9]{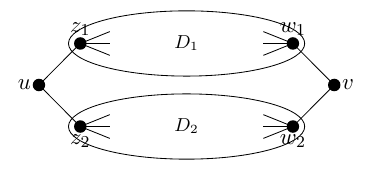}\\
\caption{An illustration for (Case 2) of Claim~\ref{claim:disconnected}. }\label{fig:proof:triangle-free:Case2}
    \end{figure}

Since $G$ is not a cycle, one of $D_1$ and $D_2$, say $D_2$, is not a path joining $z_2$ and $w_2$.
Hence $D_2$ has a spanning tree $T$ with a pendant vertex $x$ with $x\not\in \{w_2, z_2\}$. Then $x$ is not a cut vertex of $D_2$ and so $D_2-x$ is connected.
Let $V_1=\{u\}$, $V_2=D_1$, $V_3=\{v\}$, $V_4=D_2\setminus \{x\}$, and $V_5=\{x\}$.
Clearly, $V_2$ and $V_4$ induce connected graphs and $|E_G(V_1,V_2)|=|E_G(V_2,V_3)|=1$.
By contracting $V_j$'s, $G$ is $H_2$-contractible for the graph $H_2$ in Figure~\ref{fig:h1h2h3} and so $G\not\in \G$ by Lemma~\ref{prop:contractible:C4+}, which is a contradiction. We have completed the proof of Claim~\ref{claim:disconnected}.
\end{proof}

For each pair $(x,y)$ of two vertices with $\mathrm{dist}_G(x,y)=\mathrm{diam}(G)$, we denote by $N(x;y)$ the set of neighbors of $x$ which are on some shortest $(x,y)$-path, that is,
\begin{eqnarray*}
&& N(x;y) = \{ a\in N_G(x)\mid \mathrm{dist}_G(a,y)=\mathrm{diam}(G)-1\}.
\end{eqnarray*}
For simplicity, let
$n_{(x,y)}=|N(x;y) |+|N(y;x)|$.
Suppose that we take two vertices $u$ and $v$ with $\mathrm{dist}_G(u,v)=\mathrm{diam}(G)$ so that (1) $n_{(u,v)}$ is minimum and (2) $\deg_G(u)+\deg_G(v)$ is minimum subject to the condition (1).

\begin{clm}\label{claim:disconnected2}
Suppose that $H_a$ is disconnected for some $a\in\{u,v\}$.
Then there is a partition $\{A_1,A_2,A_3\}$ ($A_i\neq \emptyset$ for each $i$) of $V(G)$ satisfying all of the following (See Figure~\ref{fig:proof:Claim:triangle-free:V1V2V3}.):
\begin{itemize}
\item[\rm(i)] $A_2=N_G[a]$ (and therefore $G[A_2]$ is connected);
\item[\rm(ii)] $G[A_3]$ induces a connected graph and $b\in A_3$, where $\{b\}=\{u,v\}\setminus\{a\}$;
\item[\rm(iii)] for each vertex $z\in A_1$, $N_G(z)=N_G(a)$.
\end{itemize}
\end{clm}

\begin{figure}[h!]\centering\includegraphics[width=8cm,page=10]{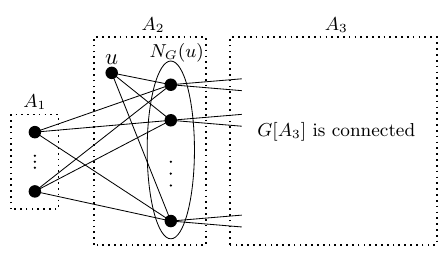}\\
\caption{An illustration for Claim~\ref{claim:disconnected2}.}\label{fig:proof:Claim:triangle-free:V1V2V3}
    \end{figure}

\begin{proof}Without loss of generality, we may assume that $a=u$.
Let $A_2=N_G[u]$, $A_3$ be the component of  $H_u$ such that $v\in A_3$, and
 $A_1=V(G)\setminus (A_2\cup A_3)$. Then (i) and (ii) follow  immediately.
We will show (iii).

Take a component $D$ of  $H_u$ other than $A_3$.
Note that every path from a vertex in $D$ to the vertex $v$ contains a vertex in $N_G(u)$.
So if there is a vertex $z\in D$ such that $E_G(z, N(u;v))=\emptyset$, then
$\mathrm{dist}(z,v)>\mathrm{dist}_G(u,v)=\mathrm{diam}(G)$,  a contradiction.
Hence, each vertex in $D$ has a neighbor in $N{(u;v)}$.
Moreover, $\mathrm{dist}(z,v)=\mathrm{dist}_G(u,v)$ and $N(z;v) \subset N(u;v)$ for all $z\in D$.
We also have $N(v;z)  \subset N{(v;u)}$ for all $z\in D$.
By the minimality of $n_{(u,v)}$, we have $N(z;v) = N(u;v)$.
Moreover, if $|D|\ge 2$, then an edge of $D$ and a vertex in $N(u;v)$ create a triangle.
Thus $|D|=1$. Let $D=\{z\}$. Since $\deg_G(z)\le \deg_G(u)$,
from the minimality of $\deg_G(u)+\deg_G(v)$, it follows that $ N_G(u)=N_G(z)$.
Therefore, it completes the proof of Claim~\ref{claim:disconnected2}.
\end{proof}

By Claim~\ref{claim:disconnected}, we may assume that $H_u$ is disconnected. Then there is a partition $\{A^{(u)}_1,A^{(u)}_2,A^{(u)}_3\}$
of $V(G)$ satisfying (i)-(iii) of Claim~\ref{claim:disconnected2},  by applying the claim for the vertex $u$.
First, suppose that $H_v$ is connected.
Then by Claim~\ref{claim:connected}, $\deg_G(v)\le 2$.
Suppose that $\deg_G(v)=1$. Let
\[V_1=A^{(u)}_1, \quad V_2=A^{(u)}_2,  \quad V_3=A^{(u)}_3\setminus N_G[v], \quad V_4=N_G(v),  \quad V_5=\{v\}.\]
Note that $G[V_2]$ and $G[V_4]$ are connected. By contracting $V_j$'s, $G$ is $H_1$-contractible for the graph $H_1$ in Figure~\ref{fig:h1h2h3}. Then  $G\not\in \G$ by Lemma~\ref{prop:contractible:path}, which is a contradiction.
Suppose that $\deg_G(v)=2$, say  $N_G(v)=\{w_1,w_2\}$. Let
\[V_1=\{w_1\},  \quad V_2=\{v\}, \quad V_3=\{w_2\}, \quad V_5=A^{(u)}_1, \quad V_4=V(G)\setminus (V_1\cup V_2\cup V_3\cup V_5).\]
Note that $G[V_4]$ induces a connected graph from the assumption that $H_v$ is connected. By contracting $V_j$'s, $G$ is $H_2$-contractible for the graph $H_2$ in Figure~\ref{fig:h1h2h3}. Then $G\not\in \G$ by Lemma~\ref{prop:contractible:C4+}, which is a contradiction.

Secondly, suppose that $H_v$ is disconnected. Then there is a partition $\{A^{(v)}_1,A^{(v)}_2,$
$A^{(v)}_3\}$ of $V(G)$ satisfying (i)-(iii) of Claim~\ref{claim:disconnected2},  by applying the claim for the vertex $v$.
Let
\[V_1=A^{(u)}_1,  \quad V_2=A^{(u)}_2, \quad V_4=A^{(v)}_2, \quad V_5=A^{(v)}_1, \quad V_3=V(G)\setminus (V_1\cup V_2\cup V_4\cup V_5).\]
We reach a contradiction by showing that the conditions of Lemma~\ref{prop:contractible:path} are satisfied.  Note that  $G[V_2]$ and $G[V_4]$ are connected. Since  $G$ is connected graph,
$V_3\neq \emptyset$. It remains to show that by contracting $V_j$'s, $G$ is $H_1$-contractible for the graph $H_1$ in Figure~\ref{fig:h1h2h3},

Since $\mathrm{dist}_G(u,v)=\mathrm{diam}(G)\ge 4$ and $V_3= A^{(u)}_3\cap A^{(v)}_3$, $E_G(V_2,V_4)=\emptyset$ and $E_G(V_1,V_5)=E_G(V_3,V_5)=\emptyset$.
If a vertex $a$ in $V_5$ is adjacent to a vertex $b$ in $V_2$, then $b$ is also a neighbor of $v$ by the property (iii) of Claim~\ref{claim:disconnected2}, which is a contradiction to $\mathrm{dist}_G(u,v)=\mathrm{diam}(G)\ge 4$. Thus $E_G(V_2,V_5)=\emptyset$.
Similarly, $E_G(V_1,V_3)=E_G(V_1,V_4)=\emptyset$, and therefore by contracting $V_j$'s, $G$ is $H_1$-contractible.
\end{proof}

\subsection{Proof of Theorem~\ref{main:thm:bipartite:CS}}

Note that the `if' part of Theorem~\ref{main:thm:bipartite:CS} follows by Theorem~\ref{thm:cycle}, {Corollary~\ref{cor:tree}, Propositions~\ref{prop:book:graph}, and~\ref{prop:case:(5)}.}
We devote this subsection to prove the `only if' part of Theorem~\ref{main:thm:bipartite:CS}.

For a bipartite graph $G=(X,Y)$, a vertex $x$ in $X$ (resp.$y$ in $Y$) is called a \textit{universal} vertex if $N_G(x)=Y$ (resp.$N_G(y)=X$). Otherwise, we say $x$ (resp.$y$) is non-universal.
Note that for an edge $xy$ of a bipartite graph $G$, both $x$ and $y$ are universal vertices if and only if $xy$ is a dominating edge.

\begin{lem}\label{lem:diam3_bipartite}
Let $G=(X,Y)$ be a connected bipartite graph.
\begin{itemize}
\item[\rm(i)] If $\diam(G)\le 3$, then any two vertices in the same partite set have a common neighbor.
\item[\rm(ii)] If $G\in \G$, $\diam(G)\le 3$ and $\delta(G)\ge 2$, then for every vertex $v$ with $\deg_G(v)\ge 3$, there is at most one component of $G-N_G[v]$ that is not a singleton.
\item[\rm(iii)] If $G\in \G$  and $|V(G)|\ge 5$, then no two vertices of degree two have the  same neighborhood.
\end{itemize}
\end{lem}

\begin{proof} Since $\diam(G)\le 3$, it is trivial to see that (i) holds.  Suppose that $G\in \G$, $\diam(G)\le 3$, $\delta(G)\ge 2$ and
let $D_0$, $D_1$, \ldots, $D_{m-1}$ be the components in $G-N_G[v]$  for a vertex $v$ with $\deg_G(v)\ge 3$.
By Corollary~\ref{cor:PropThree:closed_neighborhood}, $G-N_G[v]$ is disconnected and so $m\ge 2$.
Without loss of generality, we assume that $|D_0|\ge |D_1|\ge \cdots|D_{m-1}|$.
To show (ii), it is equivalent to show that $|D_1|=1$.
Suppose that $|D_1|\ge 2$.
Without loss of generality, let $v\in X$.
Then for each $i\in\{0,1\}$, $D_i\cap Y\neq\emptyset$ and take  a vertex $y_i\in D_i\cap Y$.
Note that $y_0$ and $y_1$ cannot have a common neighbor, a contradiction to (i).  Hence, $|D_i|=1$ for each $i\in\{1,\ldots,m-1\}$.

To show (iii),  suppose that there are two vertices of degree two, say $v_1$ and $v_3$, which have the same neighborhood. Let $N_G(v_1)=N_G(v_3)=\{v_2,v_4\}$.
Let $V_i=\{v_i\}$ for each $i\in\{1,2,3,4\}$, and let $V_5=V(G)\setminus (V_1\cup V_2\cup V_3\cup V_4)$.
Since $|V(G)|\ge 5$, $V_5\neq\emptyset$.
By contracting $V_j$'s, $G$ is either $H_2$-contractible or $H_3$-contractible for the graphs $H_2$ and  $H_3$ in Figure~\ref{fig:h1h2h3}.
If $G$ is $H_2$-contractible, then clearly $|E_G(V_1,V_2)|=|E_G(V_2,V_3)|=1$, and so $G\not\in\G$ by  Lemma~\ref{prop:contractible:C4+}, which is a contradiction.
If $G$ is $H_3$-contractible, then clearly $|V_1|=|V_2|=1$, $V_3$ is connected, and
the vertex $v_4\in V_4$ satisfies    $E_G(\{v_4\},V_3)\neq\emptyset$ and $E_G(\{v_4\},V_5)=E_G(V_4,V_5)$.
Then  $G\not\in\G$ by  Lemma~\ref{prop:contractible:K23}, which is a contradiction.
\end{proof}

\begin{proof}[Proof of the `only if' part of Theorem~\ref{main:thm:bipartite:CS}]
Suppose that there is a connected bipartite  graph $G\in \G$, none of the graphs described in (I)-(V). We take such $G$ so that $|V(G)|$ is as small as possible.
Then $G$ is neither a cycle nor a double star.
By Theorem~\ref{main:thm:triangle-free}, $G$ has diameter at most three.
Suppose that $\diam(G)\le 2$.
Then $G$ is a complete bipartite graph $G=K_{m,n}$ where $m\le n$.
If $m=1$ then $G$ is a star, a contradiction.
Suppose that $m\ge 2$. If $m\neq n$ then by Lemma~\ref{prop:contractible:Kmn}, $G\not\in \G$, a contradiction.
Thus $m=n$. If $m=2$, then $G=C_4$, a contradiction.
If $m\ge3$, then  $G=D(m-1,0;0,0)$ (a graph described in (V)), a contradiction.

Now suppose that $\diam(G)=3$.
Let $G=(X,Y)$. If either $|X|=1$ or $|Y|=1$ or $|X|=|Y|=2$, then $\diam(G)\le 2$, a contradiction.
Suppose that $|X|=2$ and $|Y|\ge 3$.
Then each vertex in $Y$ has degree at most two.
By Lemma~\ref{lem:diam3_bipartite} (iii), there is exactly one vertex $y\in Y$ of degree two and the other vertices in $Y$ are pendant.
By Lemma~\ref{lem:diam3_bipartite} (i), such all pendant vertices in $Y$ have the same neighbor in $X$, and so $G$ is a double star (a graph described in (II)) a contradiction.
The case where $|X|\ge 3$ and $|Y|=2$ is excluded  similarly. Hence, in the following, we assume that $|X|,|Y|\ge 3$.

\begin{clm}\label{claim:delta}
It holds that $\delta(G)\ge 2$ and moreover, each partite set has a vertex of degree at least three.
\end{clm}

\begin{proof}
Suppose that there is a vertex $x$ of $G$ such that $\deg_G(x)=1$. Without loss of generality, let $x\in X$.
By Proposition~\ref{prop:pendant},
$G-x\in \G$. Clearly, $G-x$ is a bipartite graph.
By minimality of $|V(G)|$, $G-x$ is one of graphs described in (I)-(V).
In addition, $2\le \mathrm{diam}(G-x) \le 3$.

\begin{itemize}
\item Suppose that $G-x$ is a cycle (a graph described in (I)).
Then $G$ is a cycle of length at least four plus a pendant vertex, which is
$H_2$-contractible for the graph $H_2$ in Figure~\ref{fig:h1h2h3}.
By Lemma~\ref{prop:contractible:C4+}, $G\not\in\G$, which is a contradiction.

\item Suppose that $G-x$ is a double star (a graph described in (II)).
Since $G$ has diameter at most three, $G$ is also a double star, a contradiction.

\item Suppose that $G-x$ is a book graph $B_n$ ($n\ge 2$) (a graph described in (III)). Let $x^*\in X$ and  $y^*\in Y$  be the universal vertices of $G-x$.
Since $n\ge 2$, we can take a 4-cycle $v_1v_2x^*y^*$ of $G-x$ such that each of $v_1$ and $v_2$ is not adjacent to $x$.
Let $V_i=\{v_i\}$ for $i\in\{1,2\}$, $V_3=\{x^*\}$, $V_5=\{x\}$, and $V_4=V(G)\setminus(V_1\cup V_2\cup V_3\cup V_5)$. By contracting $V_j$'s, $G$ is $H_2$-contractible for the graph $H_2$ in Figure~\ref{fig:h1h2h3} and $|E_G(V_1,V_2)|=|E_G(V_2,V_3)|=1$. Then $G\not\in \G$ by Lemma~\ref{prop:contractible:C4+}, which is a contradiction.

\item  Suppose that  $G-x$ is the graph $D(1,1;0,0)$ (a graph described in (IV)).
To have $\mathrm{diam}(G)=3$, we must have $D(1,1;0,1)$, which implies that $G\not\in \G$ by {Proposition~\ref{prop:case:(5)}.}

\item
Suppose that $G-x$ is either $D(m,n;p,q)$ or $D^*(m,n;p,q)$ ($m\ge n$) with $m\ge 2$, $n\neq 1$, and $p,q\ge 0$ (a  graph described in (V)).
Together with Lemma~\ref{lem:diam3_bipartite}~(i), to have $\mathrm{diam}(G)=3$, $G$ must be a graph $D(m,n;p',q')$ or $D^*(m,n;p',q')$ for some $p',q'\ge 0$ (a graph
described in (V)), a contradiction.
\end{itemize}
Hence, $G$ has no pendant vertex in $X$ and so $\delta(G)\ge 2$.
To show  the `moreover' part, suppose that all vertices of a partite set, say $X$, have degree 2.
By Lemma~\ref{lem:diam3_bipartite}~(i) and (iii), for two vertices $x_1$ and $x_2$ in $X$, we may let $N_G(x_1)=\{y_1,y_2\}$ and $N_G(x_2)=\{y_2,y_3\}$.
By Lemma~\ref{lem:diam3_bipartite}~(i), $y_1$ and $y_3$ have  a common neighbor in $X$, say $x_3$. Then $N_G(x_3)=\{y_1,y_3\}$ by our assumption.
Moreover, Lemma~\ref{lem:diam3_bipartite}~(i) implies that any two vertices of $X$ have a common neighbor in $Y$, which implies that the neighborhood of a vertex in $X$ is equal to one of $\{y_1,y_2\}$, $\{y_2,y_3\}$, and $\{y_1,y_3\}$.
Since $G$ is connected, $Y=\{y_1,y_2,y_3\}$.
If $|X|=3$, then $G$ is a cycle $x_1,y_1,x_3,y_3,x_2,y_2$ of length six, a contradiction.
If $|X|\ge 4$, then there is a vertex $x_4\in X\setminus\{x_1,x_2,x_3\}$ such that $N_G(x_4)=N_G(x_i)$ for some $i\in\{1,2,3\}$, a contradiction to  Lemma~\ref{lem:diam3_bipartite}~(iii).
Hence, each partite set has a vertex of degree at least three.
\end{proof}

By Claim~\ref{claim:delta}, we can take a vertex $x^*\in X$ with $\deg_G(x^*)\ge 3$ so that
\begin{itemize}
\item[(i)] the order of a largest component of $G-N_G[x^*]$ is as large as possible,
\item[(ii)] the number of components in $G-N_G[x^*]$ is as small as possible, subject to the condition (i),
\item[(iii)] the degree of $x^*$ is as small as possible, subject to the conditions (i) and (ii).
\end{itemize}
Let $D_0$, $D_1$, \ldots, $D_{m-1}$ be the components in $G-N_G[x^*]$.
By Corollary~\ref{cor:PropThree:closed_neighborhood}, $m\ge 2$.

\begin{clm}\label{claim:degree2}
If $D_i=\{x_i\}$ for some $i\in \{0,1,\ldots,m-1\}$, then either $\deg_G(x_i)=2$ or  $N_G(x_i)=N_G(x^*)$.
\end{clm}
\begin{proof}
Suppose that $D_i=\{x_i\}$ for some $i\in \{0,1,\ldots,m-1\}$, $\deg_G(x_i)\ge 3$ and $N_G(x_i) \neq N_G(x^*)$.
Then $N_G(x_i) \subsetneq N_G(x^*)$ and so $\deg_G(x_i)<\deg_G(x^*)$.
Note that the maximum order of a component of $G-N_G[x_i]$ is not less than that of $G-N_G[x^*]$, the number of components in $G-N_G[x_i]$ is at most $m$.
This contradicts the choice of $\deg_G(x^*)$, a contradiction. Thus the claim holds.
\end{proof}

Without loss of generality, we assume that $|D_0|\ge |D_1|\ge \cdots \ge |D_{m-1}|$.
By Lemma~\ref{lem:diam3_bipartite} (ii), $|D_i|=1$ for each $i\in \{1,\ldots,m-1\}$, and let
 $D_i=\{x_i\}$.
 In addition, it is clear that $N_G(x_i) \subseteq N_G(x^*)$ for each $i\in \{1,\ldots,m-1\}$.
Now we divide the proof into two cases according to the order of $D_0$.

\bigskip

\noindent\textbf{(Case 1)} Suppose that $|D_0|\ge 2$.
We will reach a contraction, by showing that $G$ is either $D(m,n;0,0)$ or $D^*(m,n;0,0)$ (a graph described in (V)) where $n=|D_0\cap Y|\ge 2$.
Since $D_0$ induces a connected graph, $D_0\cap Y\neq \emptyset$ and so the following claim holds.

\begin{clm}\label{claim:no-degree-two}
There is a vertex in $ D_0\cap Y$ of degree at least three.
\end{clm}

\begin{proof}
Suppose that a vertex $y$ in $D_0\cap Y$ has degree two. Let $N_G(y)=\{x'_1,x'_2\}$.
Note that $N_G(y)\subset V(D_0)$ by the definition of $D_i$'s.
Suppose that $G-N_G[y]$ is connected.
Then $G-(N_G[y]\cup \{z\})$ is also connected for some $z\in\{x_1,\ldots,x_{m-1}\}$.
Then let $V_1=\{x'_1\}$, $V_2=\{y\}$, $V_3=\{x'_2\}$, and $V_5=\{z\}$, and let $V_4=V(G) \setminus (V_1\cup V_2\cup V_3 \cup V_5)$. Note that $G[V_4]=G-(N_G[y]\cup \{z\})$ is connected.
By contracting $V_j$'s, $G$ is $H_2$-contractible for the graph $H_2$ in Figure~\ref{fig:h1h2h3}. Note that $|E_G(V_1,V_2)|=|E_G(V_2,V_3)|=1$, and
thus $G\not\in\G$ by Lemma~\ref{prop:contractible:C4+}, which is a contradiction.
Thus $G-N_G[y]$ is disconnected.
To have $\diam(G)\le 3$, every component of $G-N_G[y]$, except the component containing $x^*$, must be a singleton in $D_0\cap Y$. Moreover, as we have $\delta(G)\ge 2$, each such singleton component has the neighborhood $\{x'_1,x'_2\}$.
Then by contracting the largest component of $G-N_G[y]$ into one vertex, we obtain a complete bipartite graph $K_{2,|D_0\cap Y|+1}$. By Lemma~\ref{prop:contractible:Kmn}, $G\not\in \G$, a contradiction.
Hence, $D_0\cap Y$ has a vertex of degree at least three.
\end{proof}

\begin{figure}[h!]\centering
\includegraphics[width=8cm,page=11]{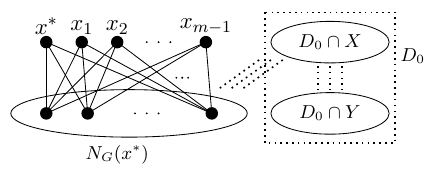}\\
\caption{An illustration for (Case 1) of the proof of Theorem~\ref{main:thm:bipartite:CS}.}\label{fig:proof:main:Case1}
    \end{figure}

\begin{clm}\label{claim:Observation_Case(i)} For each $i\in\{1,\ldots,m-1\}$, $N_G(x_i)=N_G(x^*)$. (See Figure~\ref{fig:proof:main:Case1} for an illustration.)
\end{clm}
\begin{proof} Suppose that for some $i\in\{1,\ldots,m-1\}$, $N_G(x_i)\neq N_G(x^*)$.
By Claim~\ref{claim:degree2}, $\deg_G(x_i)=2$.
By Claim~\ref{claim:no-degree-two}, we can take a vertex $y\in D_0\cap Y$ of degree at least three.
Then $G-N_G[y]$ is disconnected by Corollary~\ref{cor:PropThree:closed_neighborhood}.
Note that the component containing $V(G)\setminus D_0$ of $G-N_G[y]$ is a largest component and each of the other components of $G-N_G[y]$ is a singleton, which is a vertex in $D_0\cap Y$, by Lemma~\ref{lem:diam3_bipartite}~(ii).
Take a singleton $\{y'\}$, which is a component of $G-N_G[y]$.
Then $N_G(y')\subset N_G(y)$, and so $V(G)\setminus (N_G[x_i]\cup\{y'\})$ is connected.
Let $V_1=\{y_1\}$, $V_2=\{x_i\}$, $V_3=\{y_2\}$ where $N_G(x_i)=\{y_1,y_2\}$.
We let $V_5=\{y'\}$, and let $V_4=V(G)\setminus (V_5\cup V_1\cup V_2\cup V_3)$.
Then $G$ is $H_2$-contractible for the graph $H_2$ in Figure~\ref{fig:h1h2h3}.
By Lemma~\ref{prop:contractible:C4+}, $G\not\in\G$, a contradiction.
Hence, the claim holds.
\end{proof}

By Lemma~\ref{lem:diam3_bipartite}~(i) and Claim~\ref{claim:Observation_Case(i)},  $G$ is contractible to $K_{m+1,|N_G(x^*)|}$ with exactly one bag $D_0$ of order at least two. Hence, by Lemma~\ref{prop:contractible:Kmn}, since $G\in \mathcal{G}^c$,  it follows that $(3\le)m+1=|N_G(x^*)|$.

Now we take  a vertex  $y^*\in Y\cap D_0$ of degree at least three,
so that a largest component of $G-N_G[y^*]$ is maximum, the number of components in $G-N_G[y^*]$ is minimum, and we select one with minimum degree among such vertices.
Then by the same argument with $y^*$ instead of $x^*$ in Claims~\ref{claim:degree2} -~\ref{claim:Observation_Case(i)}, we can show that $N_G(y^*)=N_G(y')$ for all $y'\in D_0\cap Y$, and then by applying Lemma~\ref{prop:contractible:Kmn} again, we obtain $|D_0\cap Y|+1=|N_G(y^*)|$.
In addition, if $(D_0\cap X)\setminus N_G(y^*)\neq \emptyset$, then $D_0$ cannot be a connected graph, a contradiction. Hence, $N_G(y^*)=D_0\cap X$ and so $|D_0\cap Y|+1=|D_0\cap X|(\ge 3)$.
Therefore, we note that Claim~\ref{claim:last} is enough to finish the proof of (Case 1).
\begin{clm}\label{claim:last}
$G[N_G(x^*)\cup (D_0\cap X)]$ is either a complete bipartite graph or a double star.
\end{clm}

\begin{proof}
Note that, by Lemma~\ref{lem:diam3_bipartite} (i), each vertex in $N_G(x^*)$ has a neighbor in $D_0\cap X$ by considering a vertex in $N_G(x^*)$ and a vertex in $D_0\cap Y$.
By the same reason, by considering the vertex $x^*$ and a vertex in $D_0\cap X$, it holds that each vertex in $D_0\cap X$ has a neighbor in $N_G(x^*)$.

Suppose that $G[N_G(x^*)\cup (D_0\cap X)]$ is not a complete bipartite graph.
Then there is a  vertex $y\in N_G(x^*)$ such that $N_G(y)\cap D_0$ is not equal to $D_0\cap X$.
We take such $y$ with minimum degree. Let $A$ be the set of vertices in $N_G(x^*)$ that have the same neighborhood as $y$, that is, $A=\{y'\in N_G(x^*)\mid N_G(y')=N_G(y)\}$.
If $A=N_G(x^*)$, then for $G$ being connected, $N_G(y)=X$, which is a contradiction to the fact that $y$ is not universal.
Thus $A$ is a proper subset of $N_G(x^*)$ and so $|A| \le |N_G(x^*)|-1=m$.
Let $Z=V(G)\setminus (N_G[y]\cup A)$, and then let $H=G/Z$, which is the graph obtained from $G$ by contracting $Z$ into a one vertex.
Note that by the choice of $y$, every vertex in $N_G(x^*)\setminus A$ has a neighbor in $Z$ and therefore $G[Z]$ is a connected graph and $|Z|\ge 2$.
Then $H$ is a complete bipartite graph with partite sets of order $|A|+1$ and $|N_G(y)|$.
Since $G\in \G$, it follows that $|A|+1=|N_G(y)|$ from Lemma~\ref{prop:contractible:Kmn}.
Then \[m+1 \ge |A|+1=|N_G(y)| \ge m + |N_G(y)\setminus \{x^*,x_1,\ldots,x_{m-1}\}|\ge m+1,\] which implies that  $|A|=m$ and $|N_G(y)|=m+1$.
Therefore, each vertex in $A$ has exactly one neighbor $x_0$ in $N_G(y)\setminus \{x^*,x_1,\ldots,x_{m-1}\}$. Note that $\deg_G(x_0)=|N_G(x_0)\cap D_0|+|A|\ge 1+m \ge 3$.
Since $|N_G(x^*)|=m+1$, there is a unique vertex $y_0$ in $N_G(x)\setminus A$.
Then, since $\diam(G)=3$, all vertices in $(D_0\cap X) \setminus \{x_0\}$  must be adjacent to $y_0$. If $x_0y_0\not\in E(G)$, then $G-N_G[x_0]$ is connected, a contradiction to Corollary~\ref{cor:PropThree:closed_neighborhood}.
Thus $x_0y_0$ is an edge, and therefore $G[N_G(x^*)\cup (D_0\cap X)]$ is a double star.
\end{proof}

\bigskip

\noindent\textbf{(Case 2)} Suppose that $|D_0|=1$. Let $D_0=\{x_0\}$. Then $x^*$ is a universal vertex and by Claim~\ref{claim:degree2}, for every $x'\in X$, either $\deg_G(x')=2$ or $N_G(x')=N_G(x^*)=Y$.
If there is a non-universal vertex $y\in Y$ with degree at least three, then $G-N_G[y]$ has a component of order at least two, which can be shown by the same argument as in (Case 1).
Thus, we may assume that for every $y\in Y$, either $\deg_G(y)=2$ or $N_G(y)=X$.
Then, by Claim~\ref{claim:delta}, it follows that there is a universal vertex in $Y$.

If there is no degree two vertex in $X$ (or $Y$), then $G$ is isomorphic to $K_{n,n}$, and so $\diam(G)=2$, a contradiction.
Suppose that each of $X$ and $Y$ has a vertex with degree two.
Thus each part has at most two universal vertices.
We will show that each partite set has exactly one universal vertex.
Suppose that one of the partite sets, say $X$, has exactly two universal vertices $x_1$ and $x_2$.
If there are two non-universal vertices $y_1$ and $y_2$ in $Y$,  then $N_G(y_1)=N_G(y_2)=\{x_1,x_2\}$, which is a contradiction by Lemma~\ref{lem:diam3_bipartite}~(iii). Thus $Y$ has exactly one non-universal vertex, say $y_1$.
Then $\deg_G(y_1)=2$ and the vertices of $Y$ other than $y_1$ are universal vertices, and therefore $|Y|=3$.
If $|X|=3$, then $G$ is the graph $K_{3,3}$ minus an edge (a graph described in (IV)), a contradiction.
Thus $|X|\ge 4$.
Let $V_1=\{x_1\}$, $V_2=\{y_1\}$, $V_3=\{x_2\}$.
By taking a vertex $x_3\in X \setminus\{x_1,x_2\}$, let $V_5=\{x_3\}$ and $V_4=V(G)\setminus (V_1\cup V_2\cup V_3\cup V_5)$. Note that $V_4$ is connected, since the vertices in $Y\setminus \{y_1\}$, which are universal vertices, are in $V_4$ and  $V_4\cap X\neq\emptyset$. By contracting $V_j$'s,
$G$ is $H_2$-contractible for the graph $H_2$ in Figure~\ref{fig:h1h2h3}.
Then $G\not\in \G$ by Lemma~\ref{prop:contractible:C4+}, which is a contradiction.
Therefore, each partite set has exactly one universal vertex and all the other vertices have degree $2$, which implies that  $G$ is a book graph (a graph described in (III)), a contradiction. We have completed the proof.
\end{proof}

\section*{Acknowledgement}
The authors thank the referees for their valuable comments.
Shinya Fujita was supported by JSPS KAKENHI (No.~19K03603). Boram Park was upported by Basic Science Research Program through the National Research Foundation of Korea (NRF) funded by the Ministry of  Science, ICT \& Future Planning (NRF-2018R1C1B6003577).
Tadashi Sakuma was supported by JSPS
KAKENHI (No.~26400185, No.~16K05260 and No.~18K03388).


\begin{thebibliography}{10}

\bibitem{safe2018}
\'{A}gueda R., Cohen N., Fujita S., Legay S., Manoussakis Y., Matsui Y., Montero L., Naserasr R., Ono H., Otachi Y., Sakuma T., Tuza Zs., Xu R.,
Safe sets in graphs: Graph classes and structural parameters. \textit{Journal of Combinatorial Optimization}, \textbf{36} (2018), 1221-1242.

\bibitem{wsf}
Bapat R. B., Fujita S., Legay S., Manoussakis Y., Matsui Y., Sakuma T. and Tuza Zs.,
Weighted safe set problem on trees, \textit{Networks}, \textbf{71} (2018), 81--92.

\bibitem{b}
Belmonte, R., Hanaka T., Katsikarelis, I., Lampis M., Ono H., Otachi Y.,
Parameterized complexity of safe set,
\textit{ArXiv}:1901.09434 (2019).

\bibitem{cl}
Chartrand G., Lesniak L. and Zhang P., \textit{Graphs and Digraphs},
5th ed., Chapman and Hall, London, 2011.



\bibitem{ehard}
Ehard S. and Rautenbach D.,
Approximating connected safe sets in weighted trees,
\textit{ArXiv}:1711.11412v2 (2017).

\bibitem{SF2018}
Fujita S. and Furuya M.,
Safe number and integrity of graphs,
\textit{Discrete Applied Mathematics}, \textbf{247} (2018),  398--406

\bibitem{FPS:path:cycle}
Fujita S., Jensen T., Park B., and Sakuma T., On weighted safe set problem on paths and cycles, \textit{Journal of Combinatorial Optimization}, \textbf{37} (2019), 685--701.



\bibitem{sf}
Fujita S., MacGillivray G. and Sakuma T.,
Safe set problem on graphs,
\textit{Discrete Applied Mathematics}, \textbf{215} (2016), 106--111.

\bibitem{isaac}
Fujita S., Park B., and Sakuma T., Stable networks and connected safe set problem, \textit{Manuscript}.


\bibitem{chordal}
Heggerners P., van't  Hof P., L\'{e}v\^{e}que B., Paul C.,
Contracting chordal graphs and bipartite graphs to paths and trees, \textit{Discrete Applied Mathematics}, \textbf{164} (2014), 444--449.

\bibitem{kkp}
Kang B., Kim S-R., Park B., On the safe sets of Cartesian product of two complete graphs, \textit{Ars Combinatoria}, \textbf{141} (2018), 243--257.

\bibitem{chordal:dominating}
Kratsch D., Damaschke P., and Lubiw A.,
Dominating cliques in chordal graphs,
\textit{Discrete Mathematics}, \textbf{128} (1994), 269--275.

\bibitem{kv}
Korte B. and Vygen J., \textit{Combinatorial Optimization}, 5th ed.,
Springer-Verlag Berlin Heidelberg, 2011.

\end{thebibliography}
\end{document}